\documentclass[a4paper,12pt,reqno]{amsart}
\usepackage[utf8]{inputenc}
\usepackage{amssymb}
\usepackage{amsmath}
\usepackage{amsthm}
\usepackage{amsfonts}
\usepackage{mathtools}
\usepackage{makecell}
\usepackage{multirow}
\usepackage{hyperref}
\usepackage{cleveref}
\usepackage{todonotes}
\usepackage{comment}
\usepackage{lipsum}
\usepackage{mathrsfs}
\usepackage{enumitem}
\usepackage{todonotes}
\usepackage{verbatim}
\usepackage[title]{appendix}
\makeatletter
\def\dicho#1{\expandafter\@dicho\csname c@#1\endcsname}
\def\@dicho#1{\ifnum#1>1 or \else\fi(\@Roman#1)}
\makeatother
\AddEnumerateCounter{\dicho}{\@dicho}{or (III)}
\newlist{dichotomy}{enumerate}{1}
\setlist[dichotomy]{label=\dicho*,leftmargin=1.5cm}

\usepackage[british]{babel}
\usepackage[margin=1.2in]{geometry}
\usepackage{eqparbox}
\newcommand{\eqmathbox}[2][M]{\eqmakebox[#1]{$\displaystyle#2$}}
\usepackage{pbox}

\usepackage{tikz} 
\usepackage{adjustbox}
\usepackage{pgfplots}

\usetikzlibrary {arrows.meta,bending}
\definecolor{color0}{HTML}{000000}
\definecolor{color1}{HTML}{000000}
\definecolor{bgColor}{HTML}{ffffff}

\renewcommand{\O}{\mathcal{O}} 

\newcommand{\nequiv}{\not\equiv}
\newcommand{\set}[1]{\left\lbrace #1 \right\rbrace}

\newcommand{\field}[1]{\mathbb{#1}}  
\newcommand{\Q}{\field{Q}} 
\newcommand{\Z}{\field{Z}} 
\newcommand{\F}{\field{F}} 

\newcommand{\G}{\field{G}}

\renewcommand{\P}{\field{P}}

\newcommand{\Mod}[1]{\ (\mathrm{mod}\ #1)}
\newcommand{\Nm}{\textup{Nm}}
\newcommand{\fp}{\mathfrak{p}}

\newcommand{\fq}{\mathfrak{q}}

\newcommand{\legendre}[2]{\left(\frac{#1}{#2}\right)}
\newcommand{\lcm}{\textup{lcm}}

\newcommand{\minor}[1]{}
\newcommand{\main}[1]{}

\DeclareMathOperator{\trace}{Tr}

\DeclareMathOperator{\Frob}{Frob}
\DeclareMathOperator{\frob}{Frob}
\DeclareMathOperator{\Ver}{Ver}

\DeclareMathOperator{\Hom}{Hom}
\DeclareMathOperator{\Aut}{Aut}

\DeclareMathOperator{\Spec}{Spec}

\DeclareMathOperator{\Gal}{Gal}

\newcommand{\eps}{\varepsilon} 

\DeclareMathOperator{\BADFI}{\textup{\textsf{BadFI}}}

\DeclareMathOperator{\IsogPrimeDeg}{\textup{\textsf{IsogPrimeDeg}}}
\DeclareMathOperator{\IsogPrimeDegNonCM}{\textup{\textsf{IsogPrimeDegNonCM}}}
\DeclareMathOperator{\IsogPrimeDegNonCMUniform}{\textup{\textsf{IsogPrimeDegNonCM}}}

\newcommand{\TypeTwoNotMomoseBound}{\textup{\textsf{TypeTwoNotMomoseBound}}}
\newtheorem{lemma}{Lemma}
\newtheorem{theorem}[lemma]{Theorem}
\newtheorem{proposition}[lemma]{Proposition}
\newtheorem{algorithm}[lemma]{Algorithm}
\newtheorem{corollary}[lemma]{Corollary}

\theoremstyle{definition}
\newtheorem{definition}[lemma]{Definition}

\newtheorem{question}[lemma]{Question}
\newtheorem{remark}[lemma]{Remark}
\newtheorem{example}[lemma]{Example}

\newtheorem*{ack}{Acknowledgements}

\numberwithin{lemma}{section}
\numberwithin{equation}{section} 
\numberwithin{figure}{section}

\title{Towards strong uniformity for isogenies of prime degree}
\author{Barinder S. Banwait}
\address{Barinder S. Banwait \\
London,
UK}
\email{barinder.s.banwait@gmail.com}

\author{Maarten Derickx}
\address{Maarten Derickx,
Department of Mathematics,
University of Zagreb,
Croatia}
\email{maarten@mderickx.nl}
\date{}

\makeatletter
\providecommand\@dotsep{5}
\renewcommand{\listoftodos}[1][\@todonotes@todolistname]{%
  \@starttoc{tdo}{#1}}
\makeatother

\subjclass[2010]
{11G05  (primary), 
11Y60,   
11G15.   
(secondary)}

\begin{document}

\begin{abstract}
Let $E$ be an elliptic curve over a number field $k$ of degree $d$ that admits a $k$-rational isogeny of prime degree $p$. We study the question of finding a uniform bound on $p$ that depends only on $d$, and obtain, under a certain condition on the signature of the isogeny, such a uniform bound by explicitly constructing nonzero integers that $p$ must divide. As a corollary we find a uniform bound on torsion points defined over unramified extensions of the base field, generalising Merel's Uniform Boundedness result for torsion.
\end{abstract}

\maketitle
 
\section{Introduction}

Let $E$ be an elliptic curve over the rational numbers $\Q$. Mazur's isogeny theorem \cite[Theorem 1]{mazur1978rational} established that, if $E$ admits a $\Q$-rational isogeny of prime degree $p$, then $p$ must be one of the following $12$ prime numbers:
\[ \left\{2, 3, 5, 7, 11, 13, 17, 19, 37, 43, 67, 163\right\}.\]

When considering generalisations of this theorem to more general number fields $k$, 
one typically studies the analogous set of \textbf{isogeny primes}\minor{1}{
\[ \IsogPrimeDeg(k) := \left\{ \begin{aligned} p \text{ prime} : &\exists\, \phi : E \to E' \text{ a $k$-isogeny between} \\ &\text{elliptic curves over $k$ of degree } p\end{aligned} \right\}. \]
}
This set will be infinite if $k$ contains the Hilbert class field of an imaginary quadratic field $L$. Indeed, in this case there exist elliptic curves over $k$ with complex multiplications rational over the base, and so any prime that splits in $L$ will necessarily be in $\IsogPrimeDeg(k)$. Under the Generalised Riemann Hypothesis (GRH), it is known \cite[Theorem 2]{larson_vaintrob_2014} that \minor{2}{the only way for $\IsogPrimeDeg(k)$ to be infinite is in this case that $k$ contains the Hilbert class field of an imaginary quadratic field}. Thus, in order to have any hope of a finiteness result for isogeny primes in general, one is required to consider only those isogeny primes which do not arise from complex multiplications defined over the base field in the case that $k$ contains the Hilbert class field of an imaginary quadratic field. Note however that this still includes isogeny primes that can be constructed from complex multiplications defined over larger fields, which explains how the primes $43$, $67$ and $163$ arise in Mazur's list; see the Introduction to \cite{mazur1978rational} for more on how these primes arise.

One may then formulate two questions as generalisations of Mazur's isogeny theorem.

\begin{question}\label{q:main}
\begin{enumerate}
    \item \label{q:uniformity} \textbf{Uniform boundedness.} For a fixed number field $k$, is the set \minor{1}{
    \[ \IsogPrimeDegNonCM(k) := \left\{ \begin{aligned} p \text{ prime} : &\exists\, \phi : E \to E' \text{ a $k$-isogeny between non-CM} \\ &\text{elliptic curves over $k$ of degree } p\end{aligned} \right\}. \]
    }
    of primes $p$ which arise as the degree of a non-CM $k$-rational isogeny as one varies over all elliptic curves over $k$ a finite set? If so, can we determine this finite set?
    \item \label{q:strong_uniformity} \textbf{Strong Uniform boundedness.} For a fixed positive integer $d$, is the set\minor{1}{
    \[
    \IsogPrimeDegNonCMUniform(d) :=
    \bigcup_{\substack{k \\ [k:\Q] = d}} \IsogPrimeDegNonCM(k)
    \]
    }
    of primes $p$ which arise as the degree of a non-CM $k$-rational isogeny as one varies over all elliptic curves over \emph{all number fields $k$ of degree} $d$ a finite set? If so, can we determine this finite set?
\end{enumerate}

\end{question}

The second question is motivated by the analogous question for torsion points instead of isogenies. This analogous question was known as the Strong Uniform Boundedness Conjecture for torsion and is now a theorem due to Merel \cite[Theorem 2]{merel1996bornes}.

The previous discussion shows that \Cref{q:main}~\ref{q:uniformity} admits, conditional on GRH, a positive answer in the case that $k$ does not contain the Hilbert class field of an imaginary quadratic field. For a general $k$, \Cref{q:main}~\ref{q:uniformity} remains open, although it is a conjecture of Momose \cite{momose1995isogenies} that \Cref{q:main}~\ref{q:uniformity} has a positive answer for all $k$. Indeed, Momose's main theorem in \emph{loc. cit.} (Theorem A) proved the existence of a finite set $C_k$ of primes, outside of which any $k$-rational $p$-isogeny for any elliptic curve over $k$ must be one of three Types, hereafter referred to as Momose Types 1, 2 and 3. (See \Cref{sec:background} or \cite[Def 3.15, 3.16, 3.20]{banwait2022derickx} for a precise definition of Momose types.) That there are only finitely many primes occurring as the degree of a $k$-rational isogeny of Momose Type 1 was proved by Momose (Theorem B) for number fields of degree up to $12$, with the condition on the degree being removed after Merel's resolution of the uniform boundedness conjecture for torsion on elliptic curves over number fields. Similarly, that there can be only finitely many isogeny primes of Momose Type 2 was proved by Momose for $k$ a quadratic field (Theorem C), and for all $k$ conditional upon GRH (Remark 8). The isogenies of Momose Type 3 can only arise in the case that $k$ contains the Hilbert class field of an imaginary quadratic field; all CM isogenies fall into this Type, and the content of Momose's conjecture is that there can be at most finitely many primes occurring as the degree of a non-CM isogeny of Momose Type 3. (\minor{3}This conjecture is expressed both in the Introduction, as well as the first sentence of Section 5, of \cite{momose1995isogenies}.) As Momose remarks in the introduction of \emph{loc. cit.}, this is open even for $k$ a quadratic field (see Momose's Theorem D in Section 5 for partial results in this direction). We also remark that the question of strong uniformity is open for every $d > 1$.

The previous work of the authors \cite{banwait2022derickx} studied the Uniformity question, and developed an algorithm (Algorithm 8.1 of \emph{loc. cit.}) conditional on GRH that, on input any number field $k$, outputs a finite set of primes which \emph{contains} the isogeny primes that are not of Momose Type 3 (see Theorem 1.7 of \emph{loc. cit.}). This algorithm built upon a previous algorithm of the first author \cite{banwait2021explicit}, which was restricted to the case of $k$ a quadratic field, and the authors used this more general algorithm to provide the first instances of the exact determination of $\IsogPrimeDeg(k)$ for $k$ a cubic field \cite[Theorem 1.8]{banwait2022derickx}.

In the present paper we study the more difficult question of Strong Uniformity\footnote{our previous work made extensive use of auxiliary prime ideals $\fq$ and their order in the class group of a given number field; this is the principal obstacle to our previous work being readily extendible to the strong uniformity question.}. In \Cref{sec:background} we recall the notion of the \textbf{signature} $\eps$ of an isogeny, which is an element of $\left\{0,4,6,8,12\right\}^d$. \main{2} In its subsection \ref{sec:signature_interpretation} we also give a conjectural description of how to interpret these signatures in terms of the geometry of modular curves. These signatures are divided into four classes: the generic signatures, and the signatures of Types 1, 2 and 3 (see \Cref{def:sig_types}). In particular, if an isogeny has Momose Type $n$ (for $n = 1,2,3$), then it also has a signature of Type $n$ (see \cite[Section 3.4]{banwait2022derickx}). The problem therefore reduces to finding, for each of the $5^d$ possible values of $\eps$, a nonzero integer $B_{\eps}$ such that, if there exists an elliptic curve over a degree $d$ number field $k$ admitting a $k$-rational $p$-isogeny with signature $\eps$, then $p | B_{\eps}$. While we are not able to do this for \emph{all} possible signatures, we are able to do this under a condition on the \textbf{trace} of $\eps$, defined as the sum of the integers in the tuple; as explained more precisely below, most signatures satisfy this trace condition. In \Cref{sec:divisibility_conditions} we describe our algorithm to compute, for a given auxiliary prime $q$, an integer $B_{\eps,q}$ which is a multiplicative bound for prime degrees of $k$-rational isogenies of signature $\eps$. If this integer is nonzero, then we obtain a meaningful bound on the possible prime-degree isogenies as desired. We also treat one case where this integer may be zero: we show that either $p$ divides a related, necessarily nonzero integer, or $p$ must split in a certain imaginary quadratic field. This is summarised in the following result, which is the main result of the paper.

\begin{theorem}\label{thm:main}
Let $k$ be a number field of degree $d$, and $E/k$ an elliptic curve admitting a $k$-rational $p$-isogeny of signature $\eps$ for $p$ prime. Write $\trace \eps$ for the sum of the integers in $\eps$. For a given prime $q$, let $B_{\eps,q}$ and $B^\ast_{\eps,q}$ be the integers output by \Cref{alg:b_ints}. (In particular, $B^\ast_{\eps,q}$ is always nonzero, and if $B_{\eps,q} \neq 0$, then $B_{\eps,q} = B^\ast_{\eps,q}$.)

\begin{enumerate}[label=(\alph*)]
    \item If $\trace \eps \nequiv 0 \Mod{6}$, then for all primes $q$, we have $B_{\eps,q} \neq 0$, $p | B_{\eps,q}$, and \[ p \leq (2^{\trace \eps} + 2^{12d})^{2^d}. \]
    \item If $\eps$ is of Type $1$ (i.e. $\trace \eps = 0$ or $12d$), then $p$ divides the nonzero and explicitly computable integer $B_{\eps,q}^{\ast\ast}$ from \Cref{prop:uniform_type_1_bound}, and in particular \[p \leq (3^{12d}+1)^{2^d}.\]
    \item If $\trace \eps \equiv 6 \Mod{12}$, then let $q \geq 5$ be a prime. Then:
    
    \begin{enumerate}[label=(\roman*)]
        \item $B_{\eps,q} = 0$.
        \item Either $p$ divides the nonzero integer $B^\ast_{\eps,q}$, or $p$ splits in $\Q(\sqrt{-q})$.
    \end{enumerate}
\end{enumerate}
\end{theorem}

\begin{remark}
    \begin{enumerate}
        \item The above theorem deals with isogeny signatures based on $\trace \eps$. This trace is always an integer between $0$ and $12d$. So the traces one still needs to deal with after the Theorem are the integers strictly between $0$ and $12d$ that are congruent to $0 \Mod 6$.
        \item Note that in case $c)$ of the theorem we do not get an absolute bound on $p$, rather just some criteria that $p$ needs to satisfy. In \Cref{appendix}, we show how to obtain an additive bound on $p$ in case (c) of the above theorem that is conditional on GRH. Thus, assuming GRH, the only remaining signature cases are those for which $\trace \eps \equiv 0 \Mod{12}$ and  $\trace \eps \neq 0, 12d$.
        \item In \Cref{sec:how_many_left} we explain how the space of $5^d$ signatures admits natural permutations, which in particular imply that one need only identify nonzero integers $B_{\eps}$ for a restricted set of signatures $\eps$.
    \end{enumerate}
\end{remark}

For $d = 2$, this reduces to $9$ signatures that need to be dealt with (see \Cref{tab:isogeny_merel_quadratic}). The above theorem deals with all but $4$ of these: $(0,6)$, $(0,12)$, $(4,8)$, and $(6,6)$ (assuming GRH: all but $3$). For $d = 3$ one is reduced to $23$ signatures (see \Cref{tab:isogeny_merel_cubic}), and the above theorem deals with all but $9$ (assuming GRH: all but $5$) of these. This analysis is done more generally in \Cref{sec:how_many_left}, and in this sense the above theorem deals with the majority of signatures as stated earlier. It is a natural question to wonder whether ``most signatures'' translates to ``most isogenies''; i.e. to study how the possible signatures are distributed among isogenies. There is currently no data available for this, and this remains a question worthy of further investigation. 

For $d = 2$, by computing the $B_{\eps,q}$ integers and taking appropriate $\gcd$s, we obtain small sets of isogeny primes for each of the $6$ signatures we can deal with multiplicatively. Moreover, in \Cref{sec:type_1_weeding} we do a more precise study of the isogeny primes of signature $(0,0)$. In this case we can fully determine the exact list of isogeny primes. This is done by using results explicitly parameterising quadratic points on $X_0(p)$ for various small values of $p$. The following will be proven in \Cref{sec:type_1_weeding}.

\begin{theorem}\label{thm:type_1_exact_list}
    There exists an elliptic curve over a quadratic field $K$ admitting a $K$-rational $p$-isogeny of signature $(0,0)$, for $p$ prime, if and only if $p$ is in the following set: \[ \left\{2, 3, 5, 7, 11, 13, 17, 19, 37, 43, 73\right\}.\]
\end{theorem}

This result may be considered a generalisation of Kamienny's uniform boundedness of prime order torsion points on elliptic curves over quadratic fields \cite{kamienny1992torsion}, which proved that only the primes $\left\{2,3,5,7,11,13\right\}$ arise as the degree of rational torsion points on such elliptic curves.

\begin{table}[htp]
\begin{center}
\begin{tabular}{|c|c|}
\hline
Signature & Possible Isogeny Primes $\geq 13$\\
\hline
$(0,0)$ & $13, 17, 19, 37, 43, 73$\\
$(0,4)$ & $17, 23, 29, 41, 53$\\
$(0,8)$ & $17, 23, 29, 41$\\
$(4,4)$ & $17, 23, 29, 41$\\
$(4,6)$ & $17, 23$\\
$(0,6)$ & Primes $\leq 1.01 \times 10^{864}$ (GRH)\\
$(0,12)$ & ? (contains 23)\\
$(4,8)$ & ? (contains 29)\\
$(6,6)$ & ? (contains 23)\\ 
\hline
\end{tabular}

\vspace{0.3cm}
\caption{\label{tab:isogeny_merel_quadratic}
Status of strong uniform boundedness of isogeny primes of elliptic curves over quadratic fields. For every isogeny of an elliptic curve over a quadratic field, either its dual or its Galois conjugate must have one of the $9$ signatures in the left-most column. The `Possible Isogeny Primes $\geq 13$' column then lists the primes $\geq 13$ that could possibly (but may in fact not) occur as the degree of such an isogeny. This list of primes for signature $(0,0)$ is exact; i.e., they are in fact all isogeny primes. The uniform boundedness for signature $(0,6)$ is only known under the assumption of GRH, and the bound there is likely far from optimal; we expect the bound here to be closer to $10^2$ than $10^{864}$. The signatures $(0,12)$, $(4,8)$ and $(6,6)$ cannot be dealt with by our methods and are the remaining obstacles to obtaining strong uniform boundedness for isogeny primes over quadratic fields. From \Cref{ex:iso_special} we do know that the prime 23 (respectively 29) must arise in the list for signature $(0,12)$ (respectively $(4,8)$). From \cite[Example 3.6]{banwait2022derickx} we know that $23$ must arise in the list for signature $(6,6)$. Note that there are infinitely many primes with signature $(0,12)$ due to CM isogenies, and the question here is whether there is a uniform bound on the non-CM isogenies of signature $(0,12)$. }
\end{center}
\end{table}

A corollary of \Cref{thm:main}(b) is the following bound on torsion primes in unramified extensions, which may be of independent interest. This corollary will be proved in \Cref{sec:divisibility_conditions}.

\begin{corollary}\label{cor:main}
Let $k$ be a number field of degree $d$, and $E/k$ an elliptic curve that obtains a rational torsion point of order $p$ over an extension $L$ of $k$ unramified at all primes dividing $p$. Then $p$ divides a nonzero and explicitly computable integer depending only on $d$ (see \Cref{prop:uniform_type_1_bound}), and \[p \leq (3^{12d}+1)^{2^d}.\]
\end{corollary}

This result may be considered a generalisation of Merel's uniform boundedness of prime order torsion points on elliptic curves over degree $d$ number fields, since by taking the extension $L$ to be just $k$ itself one recovers Merel's result. Having bounded the torsion primes as above, one may now apply a standard argument of Kamienny and Mazur to bound the not-necessarily-prime torsion orders arising in such unramified extensions for the restricted class of non-CM elliptic curves.

\begin{corollary}\label{cor:main_composite}
For every integer $d$ there exists a constant $A_d$ such that, if $k$ is a number field of degree $d$ and $E/k$ is a non-CM elliptic curve that obtains a rational torsion point of order $N$ over an extension of $k$ unramified at all primes dividing $N$, then $N \leq A_d$.
\end{corollary}

The implementations of the algorithms in this paper have been added to the aforementioned \emph{Isogeny Primes} package, whose latest release (version 2) may be found here:

\begin{center}
\url{https://github.com/isogeny-primes/isogeny-primes}
\end{center}

All filenames given in the paper will refer to files in this repository \cite{isogeny_primes}. The \path{README.md} there explains how to run our program to obtain the lists of primes as in \Cref{tab:isogeny_merel_quadratic}.

\subsection*{Connection with other work in the area}

The situation for isogeny primes over quadratic fields (summarised in \Cref{tab:isogeny_merel_quadratic}) may be compared with \cite[Theorem 1.2]{michaud2022elliptic}, where the author obtains unconditional bounds on possible isogeny primes under various restrictions on the quadratic field (namely, real quadratic fields with exponent of class group at most $2$) and the elliptic curves (namely, those semistable at all primes above the isogeny prime). \main{6} The semistability assumption simplifies the situation considerably, since here there are (up to the Galois action on signatures) only $(0,0)$ and $(0,12)$ to consider. See also remark 3.3 in \emph{loc. cit.} for further details.

The doubly exponential bound we obtain on torsion primes in unramified extensions (\Cref{cor:main}) may be compared with \cite[Theorem 1.2]{lozano2016ramification}, who obtains a linear bound that also works for prime powers, albeit at the cost of restricting the elliptic curves considered to those that have potential supersingular reduction at a prime above $p$.

We note that there is another category of results often called \emph{isogeny theorems} in the literature. These concern the question of bounding the degree of an isogeny between pairs of elliptic curves (or more generally, abelian varieties) that one already knows to be isogenous. The first such result goes back to the seminal work of Masser and W\"{u}stholz \cite{masser1993isogeny}, with a more explicit, unconditional version due to Gaudron and R\'{e}mond \cite{gaudron2014polarisations}, which gives a bound involving the semi-stable Faltings height $h_\mathcal{ F}$ of $E$. We refer to \cite[Section 4]{bilu2013rational} for specifics on the Faltings height; for our purposes, it is only important that it is some real number that depends on $E$.

\begin{theorem}[Gaudron--R\'{e}mond, 2014]
    Let $E$ be an elliptic curve defined over a number field $K$ of degree $d$. Let $E'$ be another elliptic curve, defined over $K$ and isogenous to $E$ over $\overline{K}$. Then there exists an isogeny $\psi : E \to E'$ of degree at most $10^7d^2\left(\max\left\{h_{\mathcal{F}}(E),985\right\}+4\log d\right)^2$.
\end{theorem}

This result has the following corollary, to which our results are immediately comparable.

\begin{corollary}[{Bilu--Parent--Rebolledo \cite[Proposition 4.3]{bilu2013rational}, 2013}]
    Let $E$ be a non-CM elliptic curve defined over a number field $K$ of degree $d$ and admitting a cyclic isogeny over $K$ of degree $\delta$. Then $\delta \leq 10^7d^2\left(\max\left\{h_{\mathcal{F}}(E),985\right\}+4\log d\right)^2$.
\end{corollary}

Similar to our main theorem, the only way the above corollary depends on the number field is in terms of its degree. However, a significant difference from our main result is that our results do not depend on the semi-stable Faltings height of the elliptic curve. Indeed, part (2) of \Cref{q:main} essentially asks whether a dependence on the height is actually necessary to obtain an isogeny theorem for non-CM elliptic curves. In this light, our work can be seen as a first step in the direction of the rather ambitious goal of removing the dependence on the height from the above isogeny theorem.

Related to this, one may see \cite[Theorem 1.1]{parent2018heights}, which under the assumption of Brumer's conjecture provides a bound on the height of the $j$-invariant of quadratic points of $X_0(p)$ of the form $O(p^5\log p)$ not coming from $X_0^+(p)(\Q)$. This bound is independent of the particular quadratic field. Note that our results would imply a $O(1)$ height bound for quadratic points on $X_0(p)$ with certain isogeny signatures, as there simply are no quadratic points with these signatures for $p \gg 0$.

We also note that Gaudron and R\'{e}mond have recently improved upon their work on isogenies \cite{gaudron2023nouveaux}, so the above bound can likely be sharpened using the most updated version of their work.

\subsection*{Layout of the paper}

In \Cref{sec:background} we briefly recall some background material on isogeny characters, signatures, Frobenius action on cusps, and formal immersions. \main{2}In \Cref{sec:divisibility_conditions} we give a conjectural description on how to interpret isogeny signatures in terms of the geometry of modular curves. We provide a divisibility condition in \Cref{sec:most_general_divisibility} that improves upon our previous work, and in \Cref{sec:divisibility_conditions} use this condition to construct the integers $B_{\eps,q}$ and then to prove \Cref{thm:main}, \Cref{cor:main} and \Cref{cor:main_composite}. \Cref{thm:type_1_exact_list} is proved in \Cref{sec:type_1_weeding}. In \Cref{sec:how_many_left} we explain why \Cref{thm:main} treats ``most'' signature types, give some results of running our implementation for small values of $d$, and in degrees $2$ and $3$ explicitly identify the remaining signatures to be resolved in order to have full resolution of the Strong Uniformity question for these degrees.

As mentioned after the statement of \Cref{thm:main} above, one may obtain further results on bounding isogeny primes that are conditional upon GRH. To make it very clear what relies on GRH and what doesn't, we have collected these and other conditional results in \Cref{appendix}. In particular, we emphasise that the main body of the paper contains no result that depends on GRH.

\ack{
The first author was supported by the Simons Foundation grant \#550023 for the Collaboration on Arithmetic Geometry, Number Theory, and Computation. The second author is supported by the project “Implementation of cutting-edge research
and its application as part of the Scientific Center of Excellence for Quantum and Complex Systems, and Representations of Lie Algebras”, PK.1.1.10.0004, European Union,
European Regional Development Fund and by the Croatian Science Foundation under the
project no. IP-2022-10-5008.
 We are grateful to Pete L. Clark, Filip Najman, and Michael Stoll for comments on an earlier version of the manuscript. Additionally we would like to thank the referee for the very thorough review of the manuscript, leading to a much better exposition.
}

\section{Background}\label{sec:background}

We briefly recall the key ideas that we need in the sequel. More details may be found in \cite[Section 2]{banwait2021explicit} and \cite[Section 3]{banwait2022derickx}.

\subsection{Isogeny characters and signatures.}
Let $k$ be a number field of degree $d$ over $\Q$, and let $E$ be an elliptic curve over a number field $k$ admitting a $k$-rational isogeny of degree $p$. The $G_k := \Gal(\overline{k}/k)$-action on the kernel $W(\overline{k})$ of the isogeny gives rise to the \textbf{isogeny character}:
\[ \lambda : G_k \to \Aut(W(\overline{k})) \cong \F_p^\times. \]

The study of the isogeny character for the purpose of bounding isogenies goes back at least to Mazur \cite[Section 5]{mazur1978rational}, and was subsequently developed by Momose \cite{momose1995isogenies} and David \cite{david2008caractere}. The following description of $\lambda^{12}$ is a mild generalisation of a result of David (which in turn was a reformulation of a lemma of Momose) which appeared as Proposition 3.4 in \cite{banwait2022derickx}.

\begin{proposition}[Generalisation of Proposition 2.6 in \cite{david2012caractere}]\label{prop:momose_1_non_galois_md}
Let $k$ be a number field, $K$ its Galois closure and $\lambda$ a $p$-isogeny character over $k$. Then for every prime ideal $\fp_0$ lying above $p$ in $K$ there exists a formal sum $$\eps=\eps_{\fp_0}=\sum_{\sigma \in \Hom(k,K)}a_\sigma\sigma$$ with all $a_\sigma \in \left\{0,4,6,8,12\right\}$ such that for all $\alpha \in k^\times$ that are integral $\fp$-adic units at all primes $\fp$ above $p$, one has:
\[ \lambda^{12}((\alpha)) \equiv \alpha^\varepsilon \Mod{\fp_0}.\]
Furthermore if $p > 13 $ and $p$ is unramified in $k$, then for every $\fp_0$ there is a unique such signature $\eps_{\fp_0}$.
\end{proposition}

In the above statement we have used the ideal theoretic description of class field theory to see $\lambda^{12}$ as a character on the fractional ideals coprime to $p$ (cf. \cite[Remark 2.1]{banwait2022derickx}). This is done largely so that one may employ the following convenient shorthand notation $\lambda^{12}((\alpha))$:
$$\lambda^{12}((\alpha)) = \prod_{\fq} \lambda^{12}(\frob_\fq)^{v_\fq(\alpha)}.$$

We refer to $\eps_{\fp_0}$ as \textbf{an isogeny signature of $\lambda$ with respect to $\fp_0$}. Note that since $\Gal(K/\Q)$ acts transitively on the primes above $p$, a different choice of prime $\fp_0$ merely permutes the integers $a_{\sigma}$. We will therefore often drop the subscript $\fp_0$ and speak of $\eps$ as \textbf{an isogeny signature of $\lambda$}.
\begin{remark}\main{5}
The use of the word `signature' in this context is a bit unfortunate. Indeed, an isogeny character can (if $p \leq 13$ or $p$ ramifies in $k$) have multiple signatures attached to it, while the word signature itself hints at something that is unique for every isogeny character. The historical reason for this name traces back to the article \cite{freitas2015criteria}. In section 2 of that article they explicitly state they assume $p \geq 17$ and $p$ unramified causing the signature to be unique. However, in \cite{banwait2022derickx} there was a significant effort to remove the unramified assumption, at the expense of having signatures that may no longer be unique. The role a signature of an isogeny character plays in the entire story is to get an explicit description of the action of Galois of the character in terms of class field theory. The statements that are proved in this article do not require $p$ to be unramified in $k$. The way \Cref{thm:main}, \Cref{prop:uniform_type_1_bound} and \Cref{prop:uniform_case_a} should be interpreted is that one is free to choose any of the signatures attached to the isogeny in order to have the conditions of those statements be met. However, if one goes through the proof of \cite[Proposition 3.4]{banwait2022derickx} (see in particular Remark 3.2 there) one can see that $\trace \eps$ is the same for all possible signatures that one can attach to a fixed isogeny character provided $p - 1 > 12d$. So the freedom one gets from this is rather limited.
\end{remark}

We collect here some further results on isogeny signatures that will be used later in the paper.

\begin{lemma}\label{lem:atkin_lehner_signature}
Let $K$ be a quadratic field, and let $\sigma$ denote the nontrivial Galois automorphism of $K$. Let $\phi : E_1 \to E_2$ be a $K$-rational isogeny of prime degree $p$ between elliptic curves defined over $K$. If $\widehat{\phi} \cong \phi^\sigma$, then the signature of $\phi$ is of the form $(a, 12 - a)$ for some $a \in {0, 4, 6, 8, 12}$.
\end{lemma}

\begin{proof}
    Denoting by $(a,b)$ the signature of $\phi$, we know that the signature of $\widehat{\phi}$ is $(12 - a, 12 - b)$, and that the signature of $\phi^\sigma$ is $(b,a)$, from where the result follows.
\end{proof}

\begin{remark}
    \begin{enumerate}
        \item Note that $\widehat{\phi} \cong \phi^\sigma$ in particular means that $E_2 \cong E_1^\sigma$ so that in fact both $E_1$ and $E_2$ are $\Q$-curves.
        \item Let $x \in X_0(p)(K)$ be the point corresponding to $\phi$. Then the condition $\widehat{\phi} \cong \phi^\sigma$ is equivalent to $w_p(x) = \sigma(x)$ from which it follows that the image of $x$ in $X_0(p)^+(K)$ actually lies in $X_0(p)^+(\Q)$. Conversely if $x \in X_0(p)(K)$ maps to some rational point $x' \in X_0(p)^+(\Q)$ and the point $x$ is honestly quadratic, meaning $x \neq \sigma(x)$, then $w_p(x) = \sigma(x)$ and the above lemma is applicable. It is through this interpretation that the above lemma will be applied in \Cref{sec:type_1_weeding}.
    \end{enumerate}
\end{remark}

\begin{lemma}\label{lem:00_sig}
    Let $p \geq 17$ be a prime and $\phi: E \to E'$ a $p$-isogeny between elliptic curves over a number field $K$. If $p$ is inert in $K$ and $E$ has semistable reduction at the unique prime of $K$ above $p$, then either $\phi$ or its dual has signature 
    $$\underbrace{(0, 0, 0,  \ldots, 0)}_{\deg K \text{zeros}}.$$
 \end{lemma}
\begin{proof}
Let $\fp$ be the unique prime of $K$ above $p$. By \cite[Proposition 3.1]{banwait2022derickx} we know that $a_\fp=0$ or $a_\fp=12$. By \cite[Corollary 3.5]{banwait2022derickx} it follows that all the coefficients $a_\sigma$ of the signature have to be equal to $a_\fp$. In particular, the signature is $(0, 0,  \ldots, 0)$ or $(12, 12,  \ldots, 12)$. In the first case $\phi$ has signature $(0, 0,  \ldots, 0)$ and in the second case its dual has signature $(0, 0,  \ldots, 0)$.
\end{proof}

\subsection{On types and Momose types}
Fixing an ordering of the set of embeddings $\Hom(k,K)$ we may regard $\eps$ as a $d$-tuple in $\left\{0,4,6,8,12\right\}^d$.

The signatures are divided into 4 types according to the definition below.

\begin{definition}[{\cite[Def. 3.12]{banwait2022derickx}}]\label{def:sig_types}
Let $\eps = \sum_{\sigma}a_\sigma\sigma$ be an isogeny signature: 
\begin{itemize}
    \item  If $\eps = 0\sum_{\sigma}\sigma$ or $\eps = 12\sum_{\sigma}\sigma$, then $\eps$ is of \textbf{Type 1}.
    \item  If $\eps = 6\sum_{\sigma}\sigma$ (or equivalently $\eps = 6\Nm(k/\Q)$), then $\eps$ is of \textbf{Type 2}.
    \item If there exists an index 2 subgroup $H \subseteq G := \Gal(K/\Q)$ with $\Gal(K/k) \subseteq H$, $L:=K^H$ is imaginary quadratic and either 
    \begin{align*}
    \eps &= 12\sum_{\sigma \in \Sigma_L}\sigma \text{\quad or}\\
    \eps &= 12\sum_{\sigma \in \Sigma \setminus \Sigma_L}\sigma,
    \end{align*} then $\eps$ is of \textbf{Type 3 with field $L$}.
    \item If $\eps$ is not of Type 1, 2 or 3 then $\eps$ is \textbf{generic}.
\end{itemize}
\end{definition}

\begin{example}\label{ex:iso_special}

If $p$ is such that $X_0(p)$ is hyperelliptic, then one can readily generate elliptic curves with a $p$-isogeny over quadratic fields, by looking at inverse images of the hyperelliptic map at rational points. The signatures of the resulting isogenies can be computed; this is carried out in \path{magma_scripts/EpsilonTypes.m}. This allows one to find examples of isogenies of certain signatures. More details can be found in \cite[Example 3.6]{banwait2022derickx}, where it was already shown that there exists an isogeny of degree $23$ with signature $(6,6)$.

\begin{enumerate}
    \item Let $k = \Q(\sqrt{-655843})$, and let $E/k$ be an elliptic curve with $j$-invariant
    \[ \frac{79616218981957632000\sqrt{-655843} - 54720002186639769600000}{27368747340080916343}. \]
    Then $E$ admits a $k$-rational $23$-isogeny of signature $(0,12)$, and $E$ does not have CM. This example arose from executing \path{print_eps_type_info(5/7, 23)}.\\
    
    \item Let $k = \Q(\sqrt{955376})$, and let $E/k$ be an elliptic curve with $j$-invariant
    \[ 17474027581855244786031360\sqrt{955376} + 17079697731257932620483472896.\]
    Then $E$ admits a $k$-rational $29$-isogeny of signature $(4,8)$. This example arose from executing \path{print_eps_type_info(11, 29)}.
\end{enumerate}
\end{example}

Aside from signatures having a type, characters can also have a type which we will define below. We refer to types of the characters as Momose types, both in order to make explicit the distinction with the signature type, as well as recognising that it was Momose who defined these particular notions in \cite{momose1995isogenies}. The present paper contains no results on Momose Type 3 characters, and furthermore the definition of Momose Type 3 is rather involved, so we decided to omit this definition here and instead refer the interested reader to \cite[Def. 3.20]{banwait2022derickx}.

We let $\chi_p$ denote the mod-$p$ cyclotomic character.
\begin{definition}\cite[Def. 3.15, 3.16]{banwait2022derickx}
Let $\lambda$ be a $p$-isogeny character. 
\begin{itemize}
    \item If $\lambda^{12}$ or $(\lambda/\chi_p)^{12}$ is everywhere unramified then $\lambda$ is of \textbf{Momose Type $1$}.
    \item  If $\lambda^{12} = \chi_p^6$ then $\lambda$ is of \textbf{Momose Type $2$}.
    \item If $\lambda$ has a generic signature, then $\lambda$ is of \textbf{generic Momose Type}.
\end{itemize}
\end{definition}

A $p$-isogeny character having Momose Type $n$ (for $n = 1,2$ or $3$) implies that its signature has type $n$ (see \cite[Section 3.4]{banwait2022derickx}). For $n = 1$ the converse also holds:

\begin{lemma}
    $\lambda$ is of Momose Type 1 if and only if it has a signature of Type 1.
\end{lemma}

\begin{proof}
Write $\mu=\lambda^{12}$. Passing to the dual isogeny replaces $\lambda$ by $\chi_p/\lambda$. This interchanges the two clauses ``$\lambda^{12}$ unramified'' and ``$(\lambda/\chi_p)^{12}$ unramified'' defining Momose Type~1, and interchanges the signatures $(0,\dots,0)$ and $(12,\dots,12)$ defining signature type $1$. Both conditions in the lemma are thus invariant under $\lambda\mapsto\chi_p/\lambda$, so it suffices to prove
\[ \mu \text{ is everywhere unramified} \iff \eps=(0,\dots,0). \]
As $\mu$ is unramified outside $p$, the left-hand side says $\mu|_{I_\fp}=1$ for every $\fp\mid p$. By \cite[Proposition 3.1]{banwait2022derickx}, $\mu|_{I_\fp}=\theta_{p-1}^{a_\fp}$ with $0\le a_\fp\le 12e_\fp<p-1$ and $\theta_{p-1}$ of order $p-1$, so $\mu|_{I_\fp}=1\iff a_\fp=0$. By the construction of $\eps$ in \cite[Proposition 3.4]{banwait2022derickx}, each $a_\fp$ is the sum of the non-negative signature integers $a_\sigma$ for $\sigma\in\bigsqcup_i S_{\fp,i}$; hence $a_\fp=0$ for all $\fp\mid p$ if and only if $a_\sigma=0$ for all $\sigma$, i.e.\ $\eps=(0,\dots,0)$.
\end{proof}

However, $\lambda$ having a signature of type $2$ is equivalent to the character $\lambda^{12}/\chi_p^6$  being an everywhere unramified character. This condition is a weaker condition than the Momose Type 2 condition $\lambda^{12}/\chi_p^6=1$.

\subsection{Action of Frobenius at the cusps}\label{sec:frobenius_action}
\minor{17} In this section we provide a preliminary result on the action of Frobenius at the cusps of $X_0(p)$ to be used in the sequel.

\begin{lemma}\label{lem:frobenius_at_cusps}
Let $p$ and $q$ be distinct primes. Let $K_q$ be a finite field extension of
$\Q_q$ with ring of integers $\O_q$ and maximal ideal $\fq$. Let $(E,G)$ be a
pair consisting of an elliptic curve $E$ over $\O_q$ with potentially
multiplicative reduction at $\fq$ together with a subgroup $G \subset E$ of
order $p$, and let $x \in X_0(p)(\O_q)$ be the corresponding point. Then the
action of $\Frob_\fq$ on $G$ is well defined up to sign, and:
\begin{enumerate}
  \item if $x$ specializes to the cusp $\infty$ modulo $\fq$, then $\Frob_\fq$
        acts on $G$ as $\pm \Nm(\fq)$;
  \item if $x$ specializes to the cusp $0$ modulo $\fq$, then $\Frob_\fq$ acts
        on $G$ as $\pm 1$.
\end{enumerate}
\end{lemma}

Here, for the labelling of the cusps $0$ and $\infty$ on $X_0(p)$ we use the conventions of \cite[\S 2]{derickx2019torsion}, where the cusp at $\infty$ is represented by a N\'eron $1$-gon with a $\mu_p$ embedding and the cusp $0$ is a N\'eron $p$-gon with an embedding of $\Z/p\Z$.

\begin{proof}
The action of Frobenius on $G$ is a priori only well defined up to the action
of inertia. By taking a quadratic twist of $E$, we may assume that $E$ has
actual (i.e.\ not merely potential) multiplicative reduction, so that inertia
acts trivially and $\Frob_\fq$ is well defined up to sign. By a further
quadratic twist we may assume that $E$ has split multiplicative reduction;
note that neither twist affects $X_0(p)$, so $x$ and its specialization are
unchanged.

Let $\mathcal E$ be a N\'eron model of $E$. Since $E$ has split multiplicative
reduction, the identity component of $\mathcal E_{\F_\fq}$ is isomorphic to
$\G_{m,\F_\fq}$. The base change of $\mathcal E$ along an unramified extension
of $\O_q$ is still a N\'eron model, and since $p$ and $q$ are coprime, $G$
acquires a generator over some unramified extension of $\O_q$. In particular
this generator extends to a point on the N\'eron model, and so $G$ extends to
a subgroup $\mathcal G$ of $\mathcal E$.

We now treat the two cases.

\emph{Case (1):} The point $x$ specializes to the cusp $\infty$ modulo $\fq$.
This corresponds to $\mathcal G_{\F_\fq}$ being contained in the identity
component of $\mathcal E_{\F_\fq}$. In this case $\mathcal G_{\F_\fq}$ is
isomorphic to $\mu_{p,\F_\fq}$, and hence $\Frob_\fq$ acts on $G$ as multiplication by
$\pm\Nm(\fq)$.

\emph{Case (2):} The point $x$ specializes to the cusp $0$ modulo $\fq$. This
corresponds to $\mathcal G_{\F_\fq}$ being isomorphic to a subgroup of the
component group of $\mathcal E$. Since Frobenius acts trivially on the
component group, $\Frob_\fq$ acts on $G$ as multiplication by $\pm 1$.
\end{proof}

\subsection{Bounding cuspidal reduction}

\main{19} Throughout this subsection we use the cusps $0$ and $\infty$ of $X_0(p^n)$
defined by the same moduli interpretation as in \Cref{sec:frobenius_action}:
the cusp $\infty$ is represented by a N\'eron $1$-gon with a $\mu_{p^n}$
embedding, and the cusp $0$ by a N\'eron $p^n$-gon with an embedding of
$\Z/p^n\Z$. For $n=1$ these are the only two cusps, but for $n>1$ they are
only two among the cusps of $X_0(p^n)$.

\Cref{thm:parent} below is a strengthening of \cite[Thm 1.6]{parent1999bornes}, however we do not claim any originality for this result since all the ingredients of the proof are already in Parent's article.

\begin{theorem}\label{thm:parent}
Let  $E$ be an elliptic curve over a number field $k$ of degree $d$. Suppose $G \subset E$ is a cyclic subgroup scheme of prime power order $p^n$ \main{18}defined over $k$. Let $L$ be a finite field extension of $k$ such that $G(L)$ contains a point $P$ of order $p^n$. Let $q \neq 2, p$ be a prime. Suppose that for every prime $\fq'$ of $\O_L$ over $q$ the N\'eron model of $E$ over $\O_{L,\fq'}$ has split multiplicative reduction and that $P$ has order $p^n$ in the component group. Then \begin{align}
   p^n < C_p(s_p\max(3,d))^6, \label{eq:parent_bound}
\end{align}
where $C_p:=65, s_p := 2$ when $p>2$, and $C_2:=129, s_2:=3$.
\end{theorem}
Note that the condition $q \neq 2$ is not explicitly stated in \cite[Thm 1.6]{parent1999bornes}, however it is a running assumption in the entire article. Additionally, we decided to write $C_p$ for Parent's $C_p^2$.

The original statement due to Parent is obtained by taking $k=L$. The main advantage of the above slight generalisation is that one gets a smaller upper bound in $\ref{eq:parent_bound}$ when $k \neq L$.

\begin{proof}
Let $\fq = \O_k \cap \fq'$ and let $x \in X_0(p^n)(k)$ be the point corresponding to $(E,G)$. The moduli interpretation of the cusps \main{19} recalled above and the assumption that $P$ has order $p^n$ in the component group shows that $x_{L}$ reduces to the cusp $0$ modulo $\fq'$, implying that $x$ reduces to $0$ modulo $\fq$. So \Cref{thm:parent} is a consequence of case (1) of the following proposition.
\end{proof}

\begin{proposition}\label{prop:parent}
Let  $E$ be an elliptic curve over a number field $k$ of degree $d$. Suppose $G \subset E$ is a cyclic subgroup scheme of prime power order $p^n$ \main{18} defined over $k$, and let $x\in X_0(p^n)(k)$ be the corresponding point.
Let $q\neq 2,p$ be a prime such that \main{19} one of the following holds:
\begin{enumerate}
    \item for all primes $\fq$ of $\O_k$ over $q$, the point $x$ reduces to the
          cusp $0$ of $X_0(p^n)$ modulo $\fq$; or
    \item for all primes $\fq$ of $\O_k$ over $q$, the point $x$ reduces to the
          cusp $\infty$ of $X_0(p^n)$ modulo $\fq$.
\end{enumerate}
Then $$p^n < C_p(s_p\max(3,d))^6.$$
\end{proposition}

\main{19} Note that for $n>1$ the cusps $0$ and $\infty$ do not exhaust the cusps of
$X_0(p^n)$, so (1) and (2) in the above proposition are not an exhaustive dichotomy; they are the only
two cases we shall need.

Before we sketch how to deduce \Cref{prop:parent} from the results in \cite{parent1999bornes} we need to introduce some notation and definitions.

\begin{definition}
If $f : X \to Y$ is a morphism of finite type between noetherian schemes, we say that $f$ is a \textbf{formal immersion} at a point $x$ if the induced map on the completion of local rings
\[ \widehat{O}_{Y,f(x)} \to \widehat{O}_{X,x}\]
is surjective.
\end{definition}


Let $X_0(p^n)^{(d)}$ denote the $d$\textsuperscript{th} symmetric power of the modular curve $X_0(p^n)$, which we regard as a smooth scheme over $S := \Spec(\Z[1/p])$. We then have a map
\begin{alignat*}{2}
    f^{(d)}_{p^n} : \eqmathbox{X_0(p^n)^{(d)}_{/S}} &\longrightarrow \eqmathbox{J_0(p^n)_{/S}} & &\longrightarrow \eqmathbox{J_0^e(p^n)_{/S}}\\
    \eqmathbox{D} &\longmapsto \eqmathbox{[D - d(\infty)]} & &\longmapsto [D - d(\infty)] \Mod{{\mathfrak{J}}J_0(p^n)};
\end{alignat*}
here $\mathfrak{J}$ is the winding ideal (see \cite[\S II.18]{Mazur3} or \cite[\S 3.8]{parent1999bornes}) and $J_0^e(p^n)$ is the winding quotient of $J_0(p^n)$.

If $x \in X_0(p^n)(k)$ then we use $x^{(d)}$ to denote its image in $X_0(p^n)^{(d)}(\Q)$.

\begin{proof}[Proof of \Cref{prop:parent}]
For more details on this proof see \cite[\S 1.3]{parent1999bornes}. Note in case (1) one can replace $(E,G)$ by $(E/G, E[p^n]/G)$ to end up in case (2), hence we are reduced to proving case (2) of the proposition.
The assumption (2) ensures that $x^{(d)}_{\F_q}$ equals $\infty^{(d)}_{\F_q}$. Combined with $x^{(d)} \neq \infty^{(d)}$ this implies that $f_{p^n}^{(d)}$ is not a formal immersion at $\infty^{(d)}_{\F_q}$, see \cite[\S 4.12, Thm 4.15]{parent1999bornes}. Now the fact that $f_{p^n}^{(d)}$ is not a formal immersion at $\infty^{(d)}_{\F_q}$ ensures that $p^n<C_p(s_p\max(3,d))^6$, see \cite[Thm. 1.8, Prop 1.9]{parent1999bornes}.
\end{proof}

The above argument shows that the bound \ref{eq:parent_bound} in \Cref{thm:parent,prop:parent} can be improved if one can show that $f^{(d)}_{p^n}$ is a formal immersion at $\infty^{(d)}_{\F_q}$. 

\begin{definition}
Let $d,n$ be integers and $q > 2$ a prime. We  
define the set $$\BADFI(q,d,n) := \set{p \text{ prime} \mid p\neq q \text{ and }f_{p^n}^{(d)} \text{ is not a formal immersion at 
 } \infty^{(d)}_{\F_q}}.$$
\end{definition}

\Cref{thm:parent} and \Cref{prop:parent} already give a decent asymptotic upper bound in the case of cuspidal reduction. However, for the algorithmic parts of this paper, and the study of $d=2,3$, we will use the following more precise version.

\begin{theorem}\label{thm:parent_refined}
Let $E$ be an elliptic curve over a number field $k$ of degree $d$, let
$G \subset E$ be a cyclic subgroup scheme of prime power order $p^n$
defined over $k$, and let $x \in X_0(p^n)(k)$ be the corresponding point.
Let $q \neq 2,p$ be a prime, and suppose that at least one of the
following holds:
\begin{enumerate}
    \item there is a finite field extension $L$ of $k$ such that $G(L)$
          contains a point $P$ of order $p^n$, and for every prime
          $\fq'$ of $\O_L$ over $q$ the N\'eron model of $E$ over
          $\O_{L,\fq'}$ has split multiplicative reduction and $P$ has
          order $p^n$ in the component group;
    \item for all primes $\fq$ of $\O_k$ over $q$, the point $x$ reduces
          to the cusp $0$ of $X_0(p^n)$ modulo $\fq$; or
    \item for all primes $\fq$ of $\O_k$ over $q$, the point $x$ reduces
          to the cusp $\infty$ of $X_0(p^n)$ modulo $\fq$.
\end{enumerate}
Then $p \in \BADFI(q,d,n')$ for all $n'$ with $1 \leq n' \leq n$.
\end{theorem}

\begin{proof}
Fix $n'$ with $1 \le n' \le n$, and let
$x' = (E, G[p^{n'}]) \in X_0(p^{n'})(k)$ be the image of $x$ under the
degeneracy map $X_0(p^n) \to X_0(p^{n'})$, $(E,C) \mapsto (E, C[p^{n'}])$.
Each of the assumptions (1)--(3) implies that, modulo every prime $\fq$ of
$\O_k$ above $q$, the point $x$ reduces to the cusp $0$ or to the cusp
$\infty$ of $X_0(p^n)$: for (2) and (3) by hypothesis, and for (1) by the
argument in the proof of \Cref{thm:parent}. This degeneracy map fixes the
cusps $0$ and $\infty$ and commutes with reduction modulo $\fq$ (as
$q \neq p$), so $x'$ reduces to the same cusp at level $p^{n'}$. Hence $x'$
satisfies the hypotheses of \Cref{prop:parent} at level $p^{n'}$, and the
argument proving that proposition shows that $f^{(d)}_{p^{n'}}$ is not a
formal immersion at $\infty^{(d)}_{\F_q}$, i.e.\ that
$p \in \BADFI(q,d,n')$.
\end{proof}

The contents of \cite[Def 5.4, Thm 5.5]{banwait2022derickx} can be summarised as follows:
\begin{theorem}\label{thm:badfi_computable}
Let $q > 2$ be a prime and $d$ an integer. Then there is an algorithm that computes $\BADFI(q,d,1)$.
\end{theorem}
In order to use \Cref{thm:parent_refined} we also rely on the fact that the algorithm from \Cref{thm:badfi_computable} has been implemented in an efficient way in \cite{isogeny_primes}. Note that in the sequel we will only ever use \Cref{thm:parent_refined} in the $n=1$ case, but we needed the more general case to establish the strengthening of Parent's result (\Cref{prop:parent}) which yields a smaller upper bound in \cref{eq:parent_bound}.

\subsection{Isogeny signatures and moduli interpretation}\label{sec:signature_interpretation}\main{2}

In discussing our work with several mathematicians, including the referee, the following question repeatedly arose.

\begin{question}
Let $p$ be a prime and $\eps$ an isogeny signature. Does there exist a modular curve $X_\eps(p)$ whose points parametrize elliptic curves together with a $p$-isogeny of signature $\eps$?
\end{question}

We do not have a formal proof that such a modular curve does not exist. However, while working on this project we developed the intuition that the answer should be ``No'', and that the geometric meaning of isogeny signatures should instead be interpreted in a different way. What follows is a speculative discussion that attempts to capture this intuition in an informal way; it should be seen as an invitation to the interested reader to see how much of it can be made rigorous. The authors would not be surprised if some of what follows already appears somewhere in the vast literature on overconvergent modular forms.

Throughout this section we work with the following setup.
\begin{align*}
    p &: \mbox{a prime}\\
    d &: \mbox{a positive integer}\\
    k &: \mbox{a number field of degree $d$}\\
    G_k &: \Gal(\overline{k}/k)\\
    \fp_i &: \mbox{the primes of $k$ lying over $p$, for $1 \leq i \leq r$}\\
    e_i &: \mbox{the ramification index at $\fp_i$}\\
    f_i &: \mbox{the inertia degree at $\fp_i$}\\
    k_{\fp_i} &: \mbox{the completion of $k$ at $\fp_i$}\\
    \O_{\fp_i} &: \mbox{the ring of integers of $k_{\fp_i}$}\\
    I_i &: \Gal(\overline{k_{\fp_i}}/k_{\fp_i}^{nr}) \subseteq \Gal(\overline{k_{\fp_i}}/k_{\fp_i})\mbox{, the inertia subgroup at $\fp_i$}\\
    \theta_{p-1} &: I_i \to \F_p^\times\mbox{, the fundamental character (see \cite[Sec. 1]{serre_prop_galois})}\\
    E &: \mbox{an elliptic curve over $k$}\\
    W \subseteq E &: \mbox{a cyclic subgroup of order $p$}\\
    \lambda &: \mbox{the isogeny character $G_k \to \F_p^\times$ associated to $W$.}
\end{align*}

We recall the following proposition, which is the foundation on which isogeny signatures are later built.
\begin{proposition}[{\cite[Prop. 3.4]{banwait2022derickx},\cite[Prop. 3.2]{david2011borne}}]\label{prop:refined_signature}
For each prime $\fp_i$ there is an integer $0 \leq a_i \leq 12e_i$ satisfying $a_i \equiv 0, 4, 6, 8 \mod 12$ such that $$\lambda^{12}\mid_{I_i} = \theta_{p-1}^{a_i}.$$
\end{proposition}

The special fibre of $X_0(p)$ over $\Z_p$ consists of two copies of $X_0(1)$ glued together along the supersingular points, one copy corresponding to $W = \ker \Ver_p$ and the other to $W = \ker \Frob_p$. The model $X_0(p)_{\Z_p}$ is regular at every point, even at the supersingular points in characteristic $p$, with just two exceptions: if $p \equiv -1 \mod 3$ then $X_0(p)_{\Z_p}$ has a singularity of type $A_2$ at the supersingular point over $\F_p$ with $j = 0$, and if $p \equiv -1 \mod 4$ then it has a singularity of type $A_1$ at the supersingular point over $\F_p$ with $j = 1728$.

If we base change $X_0(p)$ to $\O_{\fp_i}$, the singularities at most supersingular points become singularities of type $A_{e_i-1}$. The exceptions are again $j = 0$ with $p \equiv -1 \mod 3$, where one obtains a singularity of type $A_{3e_i-1}$, and $j = 1728$ with $p \equiv -1 \mod 4$, where one obtains a singularity of type $A_{2e_i-1}$. Resolving a singularity of type $A_n$ replaces the point by a chain of $\P^1$'s of length $n$. In particular, the special fibre of the minimal regular model of $X_0(p)_{\O_{\fp_i}}$ is depicted in \Cref{fig:special_fibre}.

\begin{figure}[ht]
\centering
\begin{tikzpicture} 
\tikzset{x=1ex,y=1ex} 
\clip (20,0) rectangle (100, 58.64); 
\fill[bgColor] (0,0) rectangle (85.50, 58.64); 

\draw[line width=1.19pt, color=color0, ] (27.73, 3.17) -- (27.73, 54.68)  node[pos=1, right, scale=0.8] {$W = \ker \Ver_p, a_i=0$}; 
\draw[line width=1.19pt, color=color0, ] (57.45, 3.17) -- (57.45, 54.68)  node[pos=1, right, scale=0.8] {$W = \ker \Frob_p, a_i=24\, (=\!12e_i)$}; 

\draw[line width=1.19pt, color=color0, ] (59.43, 46.75) -- (25.75, 46.75) node[pos=0.5, above, scale=0.8] {$a_i=12$};
\draw[line width=1.19pt, color=color0, ] (59.43, 40.81) -- (25.75, 40.81) node[pos=0.5, above, scale=0.8] {$a_i=12$};
\draw[line width=1.19pt, color=color0, ] (59.43, 34.87) -- (25.75, 34.87) node[pos=0.5, above, scale=0.8] {$a_i=12$};

\draw[line width=1.19pt, color=color0, ] (35.67, 7.39) -- (25.77, 11.12) node[pos=0.5, above, scale=0.7, rotate=-19] {$a_i\!=\!4$};
\draw[line width=1.19pt, color=color0, ] (47.56, 7.39) -- (37.65, 11.12) node[pos=0.4, above, scale=0.7, rotate=-19] {$a_i\!=\!12$} ;
\draw[line width=1.19pt, color=color0, ] (59.45, 7.39) -- (49.54, 11.12) node[pos=0.4, above, scale=0.7, rotate=-20, fill=white] {$a_i\!=\!20$};
\draw[line width=1.19pt, color=color0, ] (41.62, 11.12) -- (31.71, 7.39) node[pos=0.4, below, scale=0.7, rotate=19] {$a_i\!=\!8$} ;
\draw[line width=1.19pt, color=color0, ] (53.50, 11.12) -- (43.60, 7.39) node[pos=0.4, below, scale=0.7, rotate=19] {$a_i\!=\!16$} ;

\draw[line width=1.19pt, color=color0, ] (39.57, 21.13) -- (25.70, 25.09) node[pos=0.5, above, scale=0.7, rotate=-17] {$a_i\!=\!6$} ;
\draw[line width=1.19pt, color=color0, ] (59.38, 21.13) -- (45.51, 25.09) node[pos=0.5, above, scale=0.7, rotate=-17] {$a_i\!=\!12$} ;
\draw[line width=1.19pt, color=color0, ] (49.47, 25.09) -- (35.61, 21.13)node[pos=0.5, above, scale=0.7, rotate=17] {$a_i\!=\!18$} ;

\node[color=color0, anchor=north west, scale=1, rotate=0.00, text width=6.5cm] at (61, 48.03) {Supersingular points with $j \neq 0, 1728$. These are chains of $\P^1$'s of length $e_i - 1$.};
\node[color=color0, anchor=north west, scale=1, rotate=0.00, text width=6.5cm] at (61, 26.93) {The supersingular point with $j=1728$. Only occurs when $p \equiv -1 \mod 4$. This is a chain of $\P^1$'s of length $2e_i -1$.};
\node[color=color0, anchor=north west, scale=1, rotate=0.00, text width=6.5cm] at (61, 13.11) {The supersingular point with $j=0$. Only occurs when $p \equiv -1 \mod 3$. This is a chain of $\P^1$'s of length $3e_i -1$.};
\end{tikzpicture}
\caption{The special fibre of the minimal regular model of $X_0(p)_{\O_{\fp_i}}$ in the case $e_i = 2$. Each component is labelled with the corresponding value of $a_i$.}
\label{fig:special_fibre}
\end{figure}

On the ordinary locus, the special fibre of $X_0(p)_{\O_{\fp_i}}$ consists of two components: one on which $W = \ker \Ver_p$, where $a_i = 0$, and one corresponding to $W = \ker \Frob_p$, where $a_i = 12e_i$. As described above, these two components are connected by chains of $\P^1$'s corresponding to the supersingular points. We suspect that the value of $a_i$ can be read off from the position of the point $x_i := (E,W)_{\O_{\fp_i}} \in X_0(p)(\O_{\fp_i})$ in this chain of $\P^1$'s.

\begin{question}\label{q:refined_question}
Suppose $j(x_i) \neq 0, 1728 \mod \fp_i$ and that $x_i$ specializes to the $n$-th $\P^1$ in the chain of $\P^1$'s. Is it then true that $a_i = 12n$? Does the same hold for $j = 1728$ and $j = 0$, but with $a_i = 6n$ and $a_i = 4n$ respectively?
\end{question}

This question is illustrated in \Cref{fig:special_fibre}, which shows the special fibre of the minimal regular model in the case $e_i = 2$, with each component labelled by our guess for the corresponding value of $a_i$.

A positive answer to the above question would give a precise geometric interpretation of the $a_i$ values.

\begin{definition}
The refined isogeny signature associated to $(E,W)$ is the sequence of triples $((e_1,f_1,a_1),(e_2,f_2,a_2),\ldots, (e_r,f_r,a_r))$.
\end{definition}

Note that $\sum_{i=1}^{r} e_i f_i = [k : \Q]$ and that the $a_i$ satisfy the conditions of \Cref{prop:refined_signature}.

In the proof of \cite[Prop. 3.4]{banwait2022derickx} there is a recipe for passing from the refined isogeny signature to the isogeny signature. Indeed, write $a_i = 12s_i + r_i$ with $r_i$ the remainder of $a_i$ modulo $12$, so that $s_i \leq e_i$. To the triple $(e_i,f_i,a_i)$ we then associate the sequence of integers
\[
\eps(e_i,f_i,a_i) := \underbrace{0,\ldots,0}_{f_i(e_i-s_i-1)\text{ times}},
\underbrace{r_i,\ldots,r_i}_{f_i\text{ times}},
\underbrace{12,\ldots,12}_{f_is_i\text{ times}}.
\]

Note that this does not make sense in the edge case $s_i = e_i$, in which case it should be interpreted as a sequence of $12$'s of length $f_ie_i$.

\begin{definition}
    The isogeny signature associated to the refined isogeny signature $((e_1,f_1,a_1),(e_2,f_2,a_2),\ldots, (e_r,f_r,a_r))$ is the signature obtained by concatenating $\eps(e_1,f_1,a_1)$ up to $\eps(e_r,f_r,a_r)$.
\end{definition}

\Cref{q:refined_question} is a guess for the geometric interpretation of the refined isogeny signature in terms of the corresponding point $x_i \in X_0(p)(\O_{\fp_i})$. However, a common strategy for studying degree-$d$ points on modular curves is to translate the problem into a question about $\Q$-rational points on the $d$-th symmetric power $X_0(p)^{(d)}$. This naturally leads to the following question.

\begin{question}
Does there exist a model of $X_0(p)^{(d)}$ over $\Z_p$ such that one can determine the (refined) isogeny signature associated to $(E, W)$ from the components of the special fibre of this model?
\end{question}

Even if there is no clean description in terms of the special fibres of a model of $X_0(p)^{(d)}$, we still expect that the (refined) isogeny signature can be described in terms of a stratification of $X_0(p)^{(d)}(\Q_p)$. This expectation comes from the $p$-adic nature of the definition of the isogeny signature.

We also believe that all the isogeny-character-based computations carried out in this article and in \cite{banwait2022derickx}, following the original ideas of \cite{momose1995isogenies}, can be reinterpreted in terms of \'etale descent \cite{poonen2023rational} as a method for studying rational points. Again, at present we do not have a proof justifying this belief.

To spell this out in more detail, first define $\Delta_{12} \subset (\Z/p\Z)^\times$ to be the subgroup of all elements of order dividing $12$, and set $X_{\Delta_{12}}(p) := X_1(p)/\Delta_{12}$. There is then a map $\pi_{12} : X_{\Delta_{12}}(p) \to X_0(p)$, and this map is \'etale over $\Z[1/p]$. One of the main results of descent theory \cite{poonen2023rational} then implies that
$$X_0(p)(k) = \bigcup_{\chi:\Gal(\overline k/k) \to (\Z/p\Z)^\times/\Delta_{12}} \pi_{12}^\chi \, (X_{\Delta_{12}}(p)^\chi(k)),$$
where it suffices to restrict to those characters $\chi : \Gal(\overline k/k) \to (\Z/p\Z)^\times/\Delta_{12}$ that are unramified away from $p$.

The role that the character $\chi$ plays in \'etale descent should be very similar to the role that $\lambda^{12}$, or possibly its inverse, plays in the isogeny-character computations. That is, the restrictions on how $\lambda^{12}$ acts on the inertia subgroups at $\fp_i$ coming from \Cref{prop:refined_signature} should correspond to $X_{\Delta_{12}}(p)^\chi(k_{\fp_i})$ being empty for certain characters $\chi$. Similarly, the information about $\lambda^{12}$ gained from the auxiliary primes $\fq$ of $k$ coprime to $p$ in \cite{banwait2022derickx} should correspond to $\fq$-adic obstructions to $X_{\Delta_{12}}(p)^\chi$ having $k$-rational points.

We believe that this interpretation in terms of descent, if proven correct rather than merely claimed as part of our intuition, could be a valuable way to understand the results of \cite{momose1995isogenies, banwait2022derickx} and other articles concerning isogeny signatures, as it opens the way to using techniques not based on isogeny characters in cases where the isogeny-character-based methods fail. However, what we have described so far is still lacking in light of the results of the present article: it does not explain how one can sometimes obtain results that depend only on the degree of the number field and not on the number field itself, since the descent is carried out over $k$. For interpreting the results of this article, we believe one could instead work in terms of \'etale descent on the $d$-th symmetric power $X_0(p)^{(d)}$.

If $d > 1$, then the natural map $X_{\Delta_{12}}(p)^{(d)} \to X_0(p)^{(d)}$ is not \'etale over $\Z[1/p]$ (or even over $\Q$). However, the results of \cite[Prop. 16]{arnav15} imply that there is a Galois cover $\pi^{(d)} : X \to X_0(p)^{(d)}$ with Galois group $(\Z/p\Z)^\times/\Delta_{12}$ that is \'etale over $\Z[1/p]$. We believe it should be possible to interpret \Cref{thm:main} and its proof in terms of \'etale descent applied to the cover $\pi^{(d)} : X \to X_0(p)^{(d)}$. Since we are now dealing with $\Q$-rational points and the varieties involved are defined over $\Q$, the characters occurring in the \'etale descent are characters on $\Gal(\overline \Q/\Q)$ rather than on $\Gal(\overline k/k)$. For example, we believe that the proof of part (a) of the theorem could be related to $X^{\chi}$ having local obstructions to rational points for certain characters $\chi$ and $p$ large enough. Similarly, we believe that part (b) could be interpreted as the statement that $X(\Q)$ consists entirely of points mapping to cuspidal points in $X_0(p)^{(d)}(\Q)$ when $p$ is large enough. Since $X(\Q)$ is always non-empty, such a statement cannot be proven purely in terms of local obstructions, which is consistent with the fact that the proof of part (b) is deeper and requires Kamienny and Mazur's strategy of formal immersions in addition to the arguments used in the proof of part (a).

\section{An even more general divisibility criterion}\label{sec:most_general_divisibility}

As explained in \cite[Section 3.2]{banwait2022derickx}, applying \Cref{prop:momose_1_non_galois_md} to the principal ideal $\fq^{h_\fq}$, for a choice $\fq$ of auxiliary prime ideal, allows one to define integers $A$, $B$ and $C$ that $p$ must divide, thereby yielding multiplicative bounds for $p$. We refer the reader to Definition 3.8 and Corollary 3.10 of \cite{banwait2022derickx}. While this makes the ensuing consideration of the non-zero-ness of the integers $A$, $B$ and $C$ easier to navigate, one observes that \Cref{prop:momose_1_non_galois_md} may be applied for \emph{any principal ideal}, not just those which are powers of prime ideals. By taking a general principal ideal $(\alpha)$, we can define an integer $D(\eps, \alpha)$ which also gives a multiplicative bound on $p$, which may be considered a more general divisibility criterion than \cite[Corollary 3.10]{banwait2022derickx}. To make this precise, it will be convenient to first make the following definition.

\begin{definition}
Let $q$ be a power of a rational prime and $L$ be a field of characteristic 0. We define the set
$$S(q,L) := \set{\pm 1,\pm q} \cup \set{\beta \in L \mid \beta \mbox{ is a Frobenius root over }\F_q},$$
where by a Frobenius root over $\F_q$ we mean a root of the characteristic polynomial of Frobenius $\Frob_\fq$ of any elliptic curve over $\F_\fq$.
\end{definition}

We then obtain the following result from \Cref{prop:momose_1_non_galois_md}, analogously to the construction of the $A$, $B$ and $C$ integers.

\begin{proposition}\label{prop:divisibility_bound}
Let $\lambda$ be a $p$-isogeny character over $k$ of type $\eps$ and $\alpha \in k^\times$ a $\fp$-adic unit at all primes $\fp$ above $p$. Suppose the fractional ideal $(\alpha)$ factors as $\prod_{i=1}^r \fq_i^{e_i}$. Then for each $1 \leq i \leq r$ there exists $\beta_i \in S(\Nm(\fq_i), \overline K)$ and a prime ideal $\fp_i$ of $\Q(\beta_i)$ such that
\begin{equation}
 \beta_i \equiv \lambda(\frob_{\fq_i}) \Mod{\fp_i}; 
 \end{equation}
moreover one has that $p$ divides the integer 
\begin{equation}\label{eqn:general_div}
\Nm_{\Q(\alpha^\eps,\beta_1,...,\beta_r)/\Q}\left(\alpha^\eps-\prod_{i=1}^r \beta_i^{12e_i}\right).
\end{equation}
\end{proposition}

\minor{24}{
\begin{proof}[Proof (sketch)]
Let $E/k$ be an elliptic curve giving rise to $\lambda$, and fix a prime $\mathcal{P}$ 
of $\overline{K}$ above $p$. Since no $\fq_i$ lies above $p$, considering the reduction 
type of $E$ at $\fq_i$ yields the existence of the $\beta_i$: in the potentially 
multiplicative case one may take $\beta_i \in \set{\pm 1, \pm\Nm(\fq_i)}$ 
(\cite[Proposition~1.4]{david2012caractere}, \cite[Proposition~3.3]{david2011borne}), 
and in the potentially good case one may take $\beta_i$ to be a suitable root of the 
characteristic polynomial of Frobenius at $\fq_i$ 
(\cite[Proposition~1.8]{david2012caractere}, \cite[Proposition~3.6]{david2011borne}); 
in both cases $\beta_i \equiv \lambda(\frob_{\fq_i}) \Mod{\mathcal{P}}$, and one sets 
$\fp_i := \mathcal{P} \cap \Q(\beta_i)$. Applying \Cref{prop:momose_1_non_galois_md} 
to the ideal $(\alpha)$ then gives
\[
\alpha^{\eps} \equiv \prod_{i=1}^{r} \beta_i^{12 e_i} \Mod{\mathcal{P}};
\]
both sides are units at all primes above $p$ (for the $\beta_i$, note that 
$\beta_i\overline{\beta_i} = \Nm(\fq_i)$ is prime to $p$), so taking norms down to 
$\Q$ shows that $p$ divides~\eqref{eqn:general_div}. Finally, the Weil bound leaves 
only finitely many possibilities for each $\beta_i$, so looping over all such choices 
and taking the least common multiple of the resulting values yields \eqref{eqn:general_div}, i.e., a multiplicative bound on $p$ depending only on 
$\eps$ and $\alpha$.
\end{proof}
}

We will usually take the $\lcm$ over all possible $r$-tuples $(\beta_1, \ldots, \beta_n)$ in the Cartesian product $\prod_{i=1}^{r} S(\Nm(\fq_i), \overline K)$ which yields the integer
\begin{equation}\label{eqn:general_div_practical}
D(\eps, \alpha) := \lcm \left(\Nm_{\Q(\alpha^\eps,\beta_1,...,\beta_r)/\Q}\left(\alpha^\eps-\prod_{i=1}^r \beta_i^{12e_i}\right)\right)
\end{equation}
which $p$ must also divide (see Step (3) of \Cref{alg:b_ints}). We remark that this integer could be zero, so in practice we will ignore terms in the product that are zero; i.e. we will consider the integer
\begin{equation}\label{eqn:general_div_practical_nonzero}
D^\ast(\eps, \alpha) := \lcm^\ast \left(\Nm_{\Q(\alpha^\eps,\beta_1,...,\beta_r)/\Q}\left(\alpha^\eps-\prod_{i=1}^r \beta_i^{12e_i}\right)\right).
\end{equation}
\minor{25/26}{
That this $\lcm^\ast$ is taken over a nonempty set --- equivalently, that
$D^\ast(\eps,\alpha)$ is a well-defined nonzero integer --- is the content of
the following result.

\begin{lemma}\label{lem:Dstar_nonzero}
Suppose $(\alpha) \neq (1)$. Then at least one of the terms in
\Cref{eqn:general_div_practical} is nonzero.
\end{lemma}

\begin{proof}
Since $\pm 1$ and $\pm\Nm(\fq_i)$ belong to $S(\Nm(\fq_i), \overline{K})$, we may
take $\beta_i = 1$ for all $i$, or $\beta_i = \Nm(\fq_i)$ for all $i$. The
corresponding terms in \Cref{eqn:general_div_practical} are
\[
\Nm_{\Q(\alpha^\eps)/\Q}\left(\alpha^\eps - 1\right)
\qquad \text{and} \qquad
\Nm_{\Q(\alpha^\eps)/\Q}\left(\alpha^\eps - \Nm\!\left((\alpha)\right)^{12}\right),
\]
which vanish if and only if $\alpha^\eps = 1$, respectively
$\alpha^\eps = \Nm((\alpha))^{12}$. Since $(\alpha) \neq (1)$ we have
$\Nm((\alpha))^{12} > 1$, so these conditions cannot both hold.
\end{proof}

We thereafter deal separately with the case that one of the terms in
\Cref{eqn:general_div_practical} is zero.
}

\section{Strongly uniform bounds and the proof of \texorpdfstring{\Cref{thm:main}}{the main theorem.}} \label{sec:divisibility_conditions}

We begin this section by providing an algorithm to compute the integers $B_{\eps,q}$ 
mentioned in the Introduction. The essential point is simply that when $\alpha \in \Q$ 
the integer $D(\eps, \alpha)$ from \Cref{eqn:general_div_practical} depends only on 
the splitting type of $\alpha$ in $k$ (i.e. the number $r$ of distinct prime ideals 
dividing $(\alpha)$, the ramification indices $e_i$, and the residue field degrees 
$f_i$), and consequently the integer $B_{\eps,q}$ is constructed as the $\lcm$ of the 
integers $D(\eps, q)$, one for each of the finitely many possible splitting types; 
the auxiliary integer $B^\ast_{\eps,q}$ is obtained in the same way from the 
$D^\ast(\eps, q)$.

\begin{algorithm}\label{alg:b_ints}
Given the following inputs:

\begin{itemize}
    \item an integer $d \geq 1$;
    \item a $d$-tuple $\eps \in \left\{0,4,6,8,12\right\}^d$;
    \item a rational prime $q$,
\end{itemize}
compute two integers $B_{\eps,q}$ and $B^\ast_{\eps,q}$ as follows.
\begin{enumerate}
    \item Enumerate the possible splitting types of a rational prime in a degree $d$ number field; i.e. enumerate the set of tuples
    \[ S_d := \left\{ (r,e_1,\ldots,e_r, f_1, \ldots, f_r) : 1 \leq r \leq d, \ \sum_{i=1}^{r}e_if_i = d \right\}. \]
    \item\label{step:one_norm} For each splitting type $s = (r,e_1,\ldots,e_r, f_1, \ldots, f_r)$ 
    in $S_d$, and for each $r$-tuple $\beta = (\beta_1,\ldots,\beta_r)$ in the cartesian 
    product $\prod_{i=1}^{r} S(q^{f_i}, \overline K)$, compute the norm
    \[ B_{\eps,q,s,\beta} := \Nm_{\Q(\beta_1,...,\beta_r)/\Q}\left(q^{\trace \eps}-\prod_{i=1}^r \beta_i^{12e_i}\right). \]
    \item\label{step:per_splitting_lcm} For each $s$ in $S_d$, set
    \[ B^\ast_{\eps,q,s} := \lcm^\ast_{\beta}\left( B_{\eps,q,s,\beta} \right), \]
    the least common multiple of the \textbf{nonzero} norms computed in the previous 
    step for this splitting type, and set $B^\ast_{\eps,q}$ to be the least common 
    multiple of the $B^\ast_{\eps,q,s}$ over all $s$ in $S_d$.
    \item If any of the integers in Step~\ref{step:one_norm} are zero, then set $B_{\eps,q} = 0$; otherwise set $B_{\eps,q} = B^\ast_{\eps,q}$.
\end{enumerate}
\end{algorithm}

\begin{remark}\label{rem:b_star_nonzero}
Step~\ref{step:per_splitting_lcm} is well defined: for each splitting type $s$, at 
least one of the norms $B_{\eps,q,s,\beta}$ is nonzero. Indeed, as in the proof of 
\Cref{lem:Dstar_nonzero}, the tuples $\beta = (1, \ldots, 1)$ and 
$\beta = (q^{f_1}, \ldots, q^{f_r})$ yield the terms $q^{\trace\eps} - 1$ and 
$q^{\trace\eps} - q^{12d}$ respectively, which vanish if and only if 
$\trace \eps = 0$, respectively $\trace\eps = 12d$; these cannot both hold. 
Consequently each $B^\ast_{\eps,q,s}$, and hence also $B^\ast_{\eps,q}$, is a 
nonzero integer. The intermediate integers $B^\ast_{\eps,q,s}$ reflect the fact 
that, for a fixed number field $k$, the splitting type of $q$ in $k$ is determined 
rather than chosen: in the notation of \Cref{sec:most_general_divisibility}, one has 
$B^\ast_{\eps,q,s} = D^\ast(\eps, q)$ for $s$ the splitting type of $q$ in $k$, and 
it is therefore essential that $B^\ast_{\eps,q,s}$ be nonzero for \emph{every} $s$, 
not merely for some $s$.
\end{remark}

This Algorithm is implemented in the function \path{B_eps_q} in \path{sage_code/strong_uniform_bounds.py}.

As discussed at the end of \Cref{sec:most_general_divisibility}, the integer $B_{\eps,q}$ is a multiplicative bound for $p$-isogenies of signature $\eps$, which may be zero, and hence trivial. If it is nonzero, then we obtain an additive bound on $p$ as follows.

\begin{proposition}\label{prop:gen_additive_bound}
    Let $\eps$ be an isogeny signature of degree $d$, and let $q$ be prime. Suppose that $p \mid B^*_{\eps,q}$. Then \[ p \leq (q^{\trace \eps} + q^{12d})^{2^d}. \]
\end{proposition}
Note that if  $B_{\eps,q} \neq 0$ then $B^*_{\eps,q} = B_{\eps,q}$ so that in this case we get the same bound on $p$ as well.
\begin{proof}
The condition $p \mid B^*_{\eps,q}$ implies the existence of a splitting type 
\[s = (r,e_1,\ldots,e_r,f_1,\ldots,f_r)\] in $S_d$ and a 
$\beta=(\beta_1,\ldots,\beta_r)$ as in \cref{alg:b_ints} step \ref{step:one_norm} 
such that $B_{\eps,q,s,\beta} \neq 0$ and $p \mid B_{\eps,q,s,\beta}$.
    We know that $|\beta_i| \leq q^{f_i}$ so that $\left|\prod_{i=1}^r \beta_i^{12e_i}\right| \leq q^{12d}$. Since there are at most $d$ distinct primes in $K$ dividing $q$ we also know $[ \Q(\beta_1,...,\beta_r):\Q] \leq 2^d$.  Hence $p \leq |B_{\eps,q,s,\beta}| \leq  (q^{\trace \eps} + q^{12d})^{2^d}$.
\end{proof}

In the rest of this Section we give the proof of \Cref{thm:main}. We begin with part (b).

\begin{proposition}\label{prop:uniform_type_1_bound}
Let $k$ be a number field of degree $d$, $E/k$ an elliptic curve admitting a $k$-rational $p$-isogeny of signature $\eps$ of type $1$, and $q \geq 3$ a rational prime different from $p$. Then either $p$ divides the nonzero integer $ B^\ast_{\eps,q}$ or $p \in \BADFI(q,d,1)$. In particular, \begin{align*}
p &\leq (3^{12d}+1)^{2^d} \text{ and }\\
p &\mid B_{\eps,q}^{\ast \ast} :=  B_{\eps,q}^{ \ast} \prod \BADFI(q,d,1). 
\end{align*}
\end{proposition}

\begin{proof} Recall that the Type 1 signatures are $(0,\ldots,0)$ and $(12,\ldots,12)$; and the dual isogeny swaps these signatures. Therefore without loss of generality, we may work only with $\eps = (0,\ldots,0)$ henceforth in the proof.

The integer stated is clearly nonzero, so we need to establish that either $p$ divides it or $p \in \BADFI(q,d,1)$. If we replaced $B^\ast_{\eps,q}$ with $B_{\eps,q}$, then this would follow from \Cref{prop:divisibility_bound} by taking $\alpha=q$; we thus consider the case that an integer computed in step \ref{step:one_norm} is zero. What remains is to show $p \in \BADFI(q,d,1)$ in this case. We have \[ \prod_{i=1}^r\beta_i^{12e_i} = 1\]
for some splitting type $(r,e_1,\ldots,e_r, f_1, \ldots, f_r)$. The only way this can happen is if all of the $\beta_i$ are equal to $\pm 1$; indeed all the Frobenius roots here have norm a power of $q$ \minor{27}{at least $1/2$}. In particular $E$ has potentially multiplicative reduction at all $\fq_i$. Furthermore, as shown in \Cref{lem:frobenius_at_cusps}, $\beta_i = \pm 1$ corresponds to the cusp 0; therefore writing $x$ for the corresponding $k$-point on $X_0(p)$, we have that $x$ reduces to $0$ at all $\fq_i$. In particular, condition $(1)$ of \Cref{prop:parent} is satisfied and \Cref{thm:parent_refined} shows that $p \in \BADFI(q,d,1)$.

Finally we establish the asserted additive bound on $p$ by specializing to $q=3$. If $p \mid B_{\eps,q}^*$, then we obtain $$p \leq (3^{12d}+1)^{2^d}$$ from \Cref{prop:gen_additive_bound}. Now $\max(\BADFI(3,d,1)) \leq 65(2\max(d,3))^6$, by \cite[Thm. 1.8 and Prop. 1.9]{parent1999bornes}, and $65(2\max(d,3))^6 \leq (3^{12d}+1)^{2^d}$. So if $p \in \BADFI(3,d,1)$ then $p \leq (3^{12d}+1)^{2^d}$ as well.
\end{proof}

\begin{remark}\label{rem:algo_implementation}
    Since \Cref{prop:uniform_type_1_bound} gives a multiplicative bound for every $q \geq 3$, in practice one may compute the greatest common divisor of these bounds for several $q$. This occurs in the \path{core_loop} method in \path{strong_uniform_bounds.py}, and is how the small lists of primes in \Cref{tab:isogeny_merel_quadratic} were obtained. Additionally, one may in fact make use of $q = 2$ as explained in \cite[Remark 5.11]{banwait2022derickx}, which further reduces the list of possible isogeny primes. 
\end{remark} 

We obtain the following corollary which may be considered a \emph{bound on unramified torsion primes}; this is \Cref{cor:main} from the Introduction.

\begin{corollary}\label{cor:main_in_body}
Let $d \geq 1$ be an integer, and let $E$ be an elliptic curve over a number field $k$ of degree $d$. If $E$ attains a torsion point of prime order $p$ rational over an extension of $k$ that is unramified at all primes of $k$ above $p$, then  \[p \leq (3^{12d}+1)^{2^d}.\]
\end{corollary}
\begin{proof}
\minor{28}{We may assume that $p > 12d + 1$, since otherwise the stated bound holds trivially. 
Let $L$ be the extension as in the statement. By replacing $L$ with its Galois closure 
over $k$ --- a compositum of $k$-conjugates of $L$, hence still unramified at all 
primes of $k$ above $p$ --- we may assume that $L/k$ is Galois. Let $P$ be the 
$L$-rational $p$-torsion point, and consider the subgroup $\langle P \rangle$ 
generated by $P$.

Suppose first that $\langle P \rangle$ is $k$-rational, so that we have a $k$-rational 
$p$-isogeny, whose isogeny character $\lambda$ satisfies $\sigma(P) = \lambda(\sigma)P$ 
for all $\sigma \in \Gal(\overline{k}/k)$. The fixed field of $\ker\lambda$ is $k(P)$, 
which is contained in $L$, so $\lambda$ factors through $\Gal(L/k)$. Since $L/k$ is 
unramified at all primes of $k$ above $p$, the character $\lambda$ --- and hence also 
$\mu = \lambda^{12}$ --- restricts trivially to the inertia subgroup $I_\fp$ for every 
prime $\fp$ of $k$ above $p$. On the other hand, by 
\cite[Proposition~3.1]{banwait2022derickx} (applicable since $p \geq 5$), we have 
$\mu|_{I_\fp} = \theta_{p-1}^{a_\fp}$ for some integer $0 \leq a_\fp \leq 12e_\fp$. 
Since $p - 1 > 12d \geq 12e_\fp$, the integer $a_\fp$ is uniquely determined by 
$\mu|_{I_\fp}$ by \cite[Remark~3.2]{banwait2022derickx}; as $\mu|_{I_\fp}$ is trivial, 
we conclude that $a_\fp = 0$. Thus $\lambda$ has signature $(0,\ldots,0)$, and 
\Cref{prop:uniform_type_1_bound} yields the stated bound.

If $\langle P \rangle$ is not $k$-rational, then there is some 
$\sigma \in \Gal(L/k)$ such that $P$ and $\sigma(P)$ together generate $E[p]$, whence 
$E[p] \subset E(L)$. The standard corollary of the Galois equivariance of the Weil 
pairing then gives $\zeta_p \in L$, and so the ramification index of $p$ in $L$ is at 
least $p-1$. But by assumption all of this ramification occurs in the tower $k/\Q$, 
so $p - 1 \leq d$, contradicting $p > 12d + 1$.}
\end{proof}

Having bounded the primes as in the above corollary, one may apply a standard argument of Kamienny and Mazur to bound all such torsion orders, albeit at the cost of losing the explicitness of the bound. The basic idea is explained in \cite[Section 6]{edixhoven1993rational}; slightly adapting the argument for our purposes yields the following result, which is \Cref{cor:main_composite} from the Introduction.

\begin{corollary}\label{cor:unramified_boundedness}
For every integer $d$ there exists a constant $A_d$ such that, if $k$ is a number field of degree $d$ and $E/k$ is a non-CM elliptic curve that obtains a rational torsion point of order $N$ over an extension of $k$ unramified at all primes dividing $N$, then $N \leq A_d$.
\end{corollary}

Before starting the proof of this result, we require a preparatory lemma. The following result generalises Lemma 4.6 of \cite{daniels2018torsion}.

\begin{lemma}\label{lem:isogeny_from_torsion}
    Let $E/k$ be an elliptic curve over a number field $k$, let $p$ be a prime, and let $F$ be a (possibly infinite) Galois extension of $k$ that contains only finitely many $p$\textsuperscript{th}-power roots of unity. Let $p^a$ be the order of the largest  $p$\textsuperscript{th}-power root of unity. Then:

    \begin{enumerate}
        \item there is a largest integer $g$ for which $E[p^g] \subseteq E(F)$, and moreover $g \leq a$.
        \item If $E(F)_{tors}$ contains a subgroup isomorphic to $\Z/p^g\Z \oplus \Z/p^j\Z$ with $j \geq g$, then $E$ admits a $k$-rational cyclic $p^{j - g}$-isogeny.
    \end{enumerate}
\end{lemma}

\begin{proof}
    Part (1) follows immediately from the usual consequence of the existence of the Weil pairing that if $E[p^g] \subseteq E(F)$ for some $g$, then $\Q(\zeta_{p^g})$ is a subfield of $F$; so by assumption on $F$, there is a largest such $g$.

    Part (2) is proved exactly as in \cite[Lemma 4.6]{daniels2018torsion} (where $k = \Q$), to which we refer the reader. (The restriction on $k$ is only needed for an additional conclusion about the possible values of $j - g$ which we are here not concerned with.)
\end{proof}

\begin{proof}[Proof of \Cref{cor:unramified_boundedness}]
    \Cref{cor:main} yields the finitely many primes $p_1, \ldots, p_n$ that can divide $N$, so we only need to find, for each $p_i$, an upper bound on the exponent $e_i$ for which there can be a non-CM elliptic curve that obtains a rational torsion point of order $p_i^{e_i}$ over such an extension. In the remainder of the proof we drop the subscript $i$ and work simply with $p$ and $e$.

    The first step is to construct a $k$-rational cyclic isogeny of degree a power of $p$, say $p^{h}$, from the torsion point of order $p^{e}$. We do this by applying \Cref{lem:isogeny_from_torsion}. Let $F$ denote the extension of $k$ unramified at all primes dividing $p$ over which $E$ attains a rational torsion point of order $p^e$ as in the statement of the Corollary. As in the proof of \Cref{cor:main_in_body} we may assume that $F/k$ is Galois. 

    We now show that there is a largest integer $a$ such that $\zeta_{p^a} \in F$. Indeed if $\zeta_{p^a} \in F$ then $L := k(\zeta_{p^a}) \subset F$ is a finite extension of $k$ such that $L/k$ is, \minor{29}{by definition of $F$}, unramified at all primes of $k$ above $p$. In particular the ramification at primes above $p$ in $L/\Q$ is bounded from above by $d$ and hence $(p-1)p^{a-1} \leq d$. 
    
    Therefore \Cref{lem:isogeny_from_torsion} applies, and we may write $g$ to be the largest integer for which $E[p^g] \subseteq E(F)$.  Observe that $(p-1)p^{g-1} \leq (p-1)p^{a-1} \leq d$ and hence $g$ is bounded by $\log_p(d)+1$. \minor{30}{Since $g$ is thereby bounded solely in terms of $p$ and $d$, we may assume that $e \geq g$: indeed, if $e \leq g$ then $e \leq \log_p(d) + 1$ is already bounded as required, and there is nothing to prove.}
    
    We now have both a $p^g$ and a $p^e$ torsion point in $E(F)$, whence we have a subgroup of $E(F)_{tors}$ isomorphic to $\Z/p^g\Z \oplus \Z/p^e\Z$ with $e \geq g$, and thus a cyclic $k$-rational $p^{e - g}$-isogeny. Writing $h := e - g$, we have thus constructed a $k$-rational point on $X_0(p^h)$.

    The next step is to bound $h$. From the main result of \cite{abramovich1996linear} we know that the gonality of $X_0(p^h)$ tends to infinity as $h$ does. Let $f$ be an integer such that the gonality of $X_0(p^{f})$ is strictly larger than $2d$. Then by \cite{frey1994curves} there are only finitely many points of degree $d$ on $X_0(p^{f})$. Choosing representative elliptic curves $E_1, \ldots, E_s$ for the non-CM degree $d$ points on $X_0(p^{f})$, we apply Serre's Open Image Theorem to each $E_i$ to show that there is a largest integer $m_i$ for which $E_i$ admits an isogeny of degree $p^{m_i}$. In this way we obtain the bound of $\max(f, m_1, \ldots, m_s)$ for $h$.
    
    Therefore, we may bound $e$ as $g + \max(f, m_1, \ldots, m_s)$. Since $g$ is bounded only in terms of $p$ and $d$, we obtain a bound on $e$ that depends only on $p$ and $d$, finishing the proof.
\end{proof}

\begin{remark}\main{8}
    We are not aware of any counterexamples to \Cref{cor:unramified_boundedness} for CM elliptic curves. Indeed, we think it is possible that the restriction to non-CM elliptic curves in \Cref{cor:unramified_boundedness} is an artefact of the proof (specifically in the use of Serre's Open Image Theorem), rather than a necessity for making the statement true. In particular, Shimura's reciprocity law gives very detailed information on the action of Galois on the $N$-torsion module $E[N]$ of CM elliptic curves, and this could potentially be used to prove the corollary in the CM-case.
\end{remark}

Next we prove part (a) of \Cref{thm:main}.

\begin{proposition}\label{prop:uniform_case_a}
Let $k$ be a number field of degree $d$, and $E/k$ an elliptic curve admitting a $k$-rational $p$-isogeny of signature $\eps$. Suppose that $\trace \eps \nequiv 0 \Mod{6}$. Then for all primes $q$, $p$ divides the nonzero integer $B_{\eps,q}$.
\end{proposition}

\begin{proof}
    We will show that Step (2) of \Cref{alg:b_ints} always produces a nonzero integer.

    Consider how $q$ splits in $k$: \[ q\O_k =  \prod_{i=1}^r \fq_i^{e_i}, \] with each $\fq_i$ of residue degree $f_i$. If one of the norms in Step (2) of \Cref{alg:b_ints} were zero, then we would have $(\beta_1,\ldots,\beta_r) \in \prod_{i=1}^rS(q^{f_i},\overline{\Q})$ such that
    \[q^{\trace \eps} = \prod_{i=1}^r\beta_i^{12e_i}. \]
    By considering the absolute value of this equation, and observing that the only possible values for $|\beta_i|$ are $1$, $\sqrt{q}^{f_i}$, or $q^{f_i}$, we see that $6$ must divide $\trace \eps$, contradicting our assumption.
\end{proof}

Next we consider part (c), which we shall prove after establishing the following lemma.

\begin{lemma}\label{lem:frob_roots_ss}
Let $q > 3$ be a prime, $a \equiv 6  \Mod{12}$ an integer and $e_1,\ldots,e_r$, $f_1,\ldots,f_r$ be positive integers. Let furthermore $\beta_i \in S(q^{f_i},\overline{\Q})$ be such that \[\prod_{i=1}^r \beta_i^{12e_i} = q^a.\] Then there is an $i$ for which $f_i$ is odd and $\beta_i^2=-q^{f_i}$.
\end{lemma}

\begin{proof}
Write $a = 12a'+6$; then $(q^{a'}\sqrt{-q})^{12}=q^a=\prod_{i=1}^r \beta_i^{12e_i}$, so in particular $ \zeta_{12}^t q^{a'}\sqrt{-q} = \prod_{i=1}^r \beta_i^{e_i}$ where $\zeta_{12}$ is a primitive 12th root of unity and $0 \leq t < 12$ some integer. Thus we have an equality of number fields \[ \Q(\zeta_{12}^t q^{a'}\sqrt{-q}) = \Q(\prod_{i=1}^r \beta_i^{e_i}).\] However, we also have the containment of number fields \[ \Q(\prod_{i=1}^r \beta_i^{e_i}) \subseteq \Q(\beta_1^{e_1},\cdots,\beta_r^{e_r}) \subseteq \Q(\beta_1,\cdots,\beta_r). \] Since this latter field is just the compositum of all of the $\Q(\beta_i)$, we see that $q$ ramifies in $\Q(\beta_i)$ for some $i$.

\minor{31}{
Being a Frobenius root, $\beta_i$ is a root of $x^2 + Ax + q^{f_i}$. So $q$ ramifying 
in $\Q(\beta_i)$ means that $q \mid A^2 - 4q^{f_i}$, whence $q \mid A$. By the 
contraposition of \cite[Theorem~2.5~(1)]{banwait2022derickx}, the condition $q \mid A$ 
excludes case~(1), so we are in one of the cases (2)--(5) of \cite[Theorem~2.5]{banwait2022derickx} (Waterhouse's theorem). Case~(4) requires 
$q \in \{2,3\}$ and is excluded since $q > 3$. We rule out the remaining cases other 
than (5)(i) by computing $\Q(\beta_i)$ directly:
\begin{itemize}
    \item in case~(2) we have $A = \pm 2q^{f_i/2}$, so $\beta_i = \mp q^{f_i/2} \in \Q$ 
    and $\Q(\beta_i) = \Q$;
    \item in case~(3) we have $A = \pm q^{f_i/2}$ with $f_i$ even, so 
    $A^2 - 4q^{f_i} = -3q^{f_i}$ and $\Q(\beta_i) = \Q(\sqrt{-3})$;
    \item in case~(5)(ii) we have $A = 0$ with $f_i$ even, so $\beta_i^2 = -q^{f_i}$ 
    and $\Q(\beta_i) = \Q(\sqrt{-1})$;
\end{itemize}
Since $q > 3$, in each of these cases $\Q(\beta_i)$ is unramified at $q$, contrary to 
assumption. We must therefore be in case~(5)(i) --- that is, $f_i$ is odd and $A = 0$, 
so that $\beta_i^2 = -q^{f_i}$ --- which corresponds to Case~$5.i$ of Table~4 of 
\cite{banwait2022derickx} and gives the result.
}
\end{proof}

\minor{32,33}{
\begin{proof}[Proof of part (c).]
    For part (i), write $\trace \eps = 12a' + 6$. We take the splitting type 
    $s = (d,1,\ldots,1, 1\ldots,1)$ corresponding to the completely split case. 
    Consider the $\beta$-tuple $(\beta_1, \ldots, \beta_d)$ where we choose $a'$ of 
    the $\beta$s to be $q$, one of them to be $\sqrt{-q}$, and the rest to be $1$ 
    (note that $a' < d$, since the maximum value of the trace is $12d$). For this 
    $\beta$-vector, we have that 
    $$B_{\eps,q,s,\beta} = q^{12a'+6} - \prod_{i=1}^d \beta_i^{12} = q^{12a'+6} - q^{12a'}\sqrt{-q}^{12}=0$$ 
    and hence $B_{\eps,q} = 0$.
    
For part (ii), let $a = \trace \eps$. We apply \Cref{prop:divisibility_bound} with 
    $\alpha = q$ to obtain, for each $i$, the existence of a $\beta_i \in S(q^{f_i}, 
    \overline{\Q})$ and a prime ideal $\fp_i$ of $\Q(\beta_i)$ above $p$ such that
    \begin{equation}\label{eq:lambda_beta}
        \lambda(\Frob_{\fq_i}) \equiv \beta_i \Mod{\fp_i}
    \end{equation}
     (where $\fq_i$ are the distinct prime ideals dividing the ideal $(q)$, and we 
     write $s = (r,e_1,\ldots,e_r,f_1,\ldots,f_r)$ for the corresponding splitting 
     type of $q$ in $k$) and such that, for $\beta = (\beta_1,\ldots,\beta_r)$, the 
     prime $p$ divides $B_{\eps,q,s,\beta}$. If $B_{\eps,q,s,\beta} \neq 0$, then $p$ 
     will divide $B^\ast_{\eps,q,s}$, and hence $B^\ast_{\eps,q}$. 
     
     If $B_{\eps,q,s,\beta} = 0$, then we have \[ q^a = \prod_{i=1}^r \beta_i^{12e_i}; \] 
     by applying \Cref{lem:frob_roots_ss} we have the existence of an $i$ for which 
     $f_i$ is odd and $\beta_i^2 = -q^{f_i}$, which implies $\Q(\beta_i) = 
     \Q(\sqrt{-q})$. Suppose now for a contradiction that $p$ were inert in 
     $\Q(\beta_i) = \Q(\sqrt{-q})$, so that $\fp_i = (p)$ and the residue field 
     $\mathcal{O}_{\Q(\sqrt{-q})}/\fp_i$ is isomorphic to $\F_{p^2}$. Since the isogeny 
     character $\lambda$ takes values in $\F_p^\times$, \Cref{eq:lambda_beta} shows 
     that the reduction of $\beta_i$ modulo $\fp_i$ lies in the prime subfield 
     $\F_p^\times$ of $\F_{p^2}$. Reducing the equation $\beta_i^2 = -q^{f_i}$ modulo 
     $\fp_i$ therefore gives an identity in $\F_p$, exhibiting $-q^{f_i}$ --- and hence 
     $-q$, as $f_i$ is odd --- as a square modulo $p$. This contradicts the assumption 
     that $p$ is inert in $\Q(\sqrt{-q})$. Therefore $p$ splits in $\Q(\sqrt{-q})$.
\end{proof}
}

\section{Type 1 isogenies over quadratic fields}\label{sec:type_1_weeding}

In this section we prove \Cref{thm:type_1_exact_list}. 
\minor{35}{
We execute \Cref{alg:b_ints} which is implemented as \path{uniform_isogeny_primes.py} (see also \Cref{rem:algo_implementation}) and obtain the superset \[ \set{2, 3, 5, 7, 11, 13, 17, 19, 37, 41, 43, 73, 109}; \] this is the set of primes we need to decide upon.
}

\subsection{The primes \texorpdfstring{$p \leq 13$}{p <= 13}} 
For each prime $p$ in the set $\set{2, 3, 5, 7, 11, 13}$ there exists a quadratic field $K$ and an elliptic curve $E$ over $K$ with a $K$-rational torsion point \cite{kamienny1992torsion}; the curves thus admit isogenies whose kernels are generated by points defined over an unramified extension of $K$ (i.e. the trivial extension).

\subsection{The primes \texorpdfstring{$p = 17, 19$ and $37$}{p = 17, 19 and 37}}

These are isogeny primes, as the following result shows.

\begin{proposition}
    Let $\theta_{17} := \sqrt{-648810}$, $\theta_{19} := \sqrt{-544752}$, $\theta_{37} := \sqrt{-23}$   and define the elliptic curves $E_{17}$, $E_{19}$ and $E_{37}$ by:
{\small    
\begin{align*}
         E_{17}:  y^2 &= x^3 + \frac{655305766\theta_{17} - 59206784085} {3805380}x + 
\frac{-14063682165143\theta_{17} - 439793145011475} {308235780}\\
 E_{19}:  y^2 &= x^3 + \frac{-19939113\theta_{19} - 47253873316} {186368}x + 
\frac{-1182780494221\theta_{19} - 7313237225587764} {281788416}\\
E_{37}: y^2 &= x^3 + \frac {4107\theta_{37} + 45177} {2}x + 
1114366\theta_{37} + 2228732.
\end{align*}
}
Then the curves $E_{17}, E_{19}$ and $E_{37}$ admit isogenies of signature $(0, 0)$ whose degrees are $17, 19$ and $37$ respectively.
\end{proposition}

\begin{proof}
The modular curves $X_0(17)$, $X_0(19)$ and $X_0(37)$ are (hyper)elliptic curves and hence admit infinitely many quadratic points. The above examples were found by generating quadratic points and computing their signatures. For $X_0(17)$ and $X_0(19)$ we did this by choosing a random $x$ coordinate and solving for $y$ on the elliptic models as found in Magma's Small Modular curve database. For $X_0(37)$ we did this by using the parameterisation of the quadratic points on $X_0(37)$ from \cite[Proposition 5.4]{box2021quadratic}. The Magma function \verb|print_eps_type_info| that we wrote for this can be found in \path{magma_scripts/EpsilonTypes.m} in \cite{isogeny_primes}. Note that for $X_0(17)$ and $X_0(19)$ the Magma code actually returns an isogeny whose signature is $(12, 12)$. However, their dual isogenies have signature $(0, 0)$ and these dual isogenies have $E_{17}$ and $E_{19}$ as their domain. For $X_0(37)$ the isogeny returned by the above code already had signature $(0,0)$, so taking the dual was not necessary.
\end{proof}

\minor{36}\begin{remark}
    The above elliptic curves admit a $p$-isogeny with signature $(0,0)$. In particular, if $\lambda$ is the associated isogeny character, then we know that $\lambda^{12}$ is everywhere unramified. We do not know if there are examples where $\lambda$ is itself everywhere unramified. It would be interesting to know the answer to this more refined question.
\end{remark}

\minor{35}{
\subsection{The prime \texorpdfstring{$p = 41$}{p = 41}} This prime does not arise, as the following result shows.

 \begin{proposition}\label{prop:no-41-signature-0-0}
      There is no elliptic curve over a quadratic field admitting a $41$-isogeny of signature
  $(0,0)$.
  \end{proposition}

  \begin{proof}
  If $P$ is a quadratic point on $X_0(41)$ arising from pulling back a $\Q$-point on $X_0(41)^+$, then \Cref{lem:atkin_lehner_signature} applies to show that the signature is not $(0,0)$. Thus $P$ must be an exceptional quadratic point. By work of Bruin and Najman
  \cite[Table 12]{bruin2015hyperelliptic}, there are, up to Galois conjugacy, precisely two of these, both
  defined over $\mathbb{Q}(i)$; on the model
  \[
  y^2 + (-x^4 - x)y = -x^7 - 2x^6 + 2x^5 + 5x^4 + 2x^3 - 4x^2 - 5x - 2
  \]
  of $X_0(41)$ they are given by
  \[
  P_1 = \left( \tfrac{1}{2}(-i - 1),\ \tfrac{1}{4}(-3i - 4) \right)
  \quad\text{and}\quad
  P_2 = \left( \tfrac{1}{2}(-i - 1),\ \tfrac{1}{4}(i + 1) \right).
  \]
  Using the Magma function \verb|eps_from_isogeny| in \path{magma_scripts/EpsilonTypes.m} of
  \cite{isogeny_primes} (the precise code may be found at the bottom of that file), we computed the
  signatures of the isogenies corresponding to $P_1$ and $P_2$ to be $(0,4)$ and $(12,8)$
  respectively; as a consistency check, these two isogenies are dual to one another, since $P_2 =
  w_{41}(P_1)$. The Galois conjugates $P_1^{\sigma}$ and $P_2^{\sigma}$ then yield the signatures
  $(4,0)$ and $(8,12)$. In particular, no quadratic point of $X_0(41)$ yields an isogeny of
  signature $(0,0)$.
  \end{proof}
}
\subsection{The prime \texorpdfstring{$p = 43$}{p = 43}} This prime arises. In fact we have the following more precise statement.

\begin{proposition}
\begin{enumerate} Let $K$ be the quadratic field $K=\Q(\sqrt{-71})$, let $\sigma$ be the nontrivial Galois automorphism of $K$, and define $$j_{43} := -\frac {49\sqrt{-71} +977} {4}.$$
\item An elliptic curve $E$ over a quadratic field $K'$ admits a $43$-isogeny of signature $(0,0)$ if and only if $K' = K$ and $j(E) = j_{43}$ or $\sigma(j_{43})$.
\item The elliptic curve $E_{43}$ given by $$y^2=x^3-3j_{43}(j_{43}-1728)x+2j_{43}(j_{43}-1728)^2$$ obtains a torsion point of degree $43$ over an extension of $K$ that is unramified at both primes dividing $43 \O_K$.
\item There is no elliptic curve over a quadratic field that obtains a 43-torsion point over an everywhere unramified extension of that field.
\end{enumerate}
\end{proposition}

Before we start the proof we describe the rational and quadratic points on $X_0(43)$. The rational points on $X_0(43)$ have been classified in \cite{mazur1978rational} and the quadratic points in \cite{box2021quadratic}. The curve $X_0(43)$ has 3 rational points; two of these are cusps, and one corresponds to the elliptic curve with $j$-invariant $-88473600$, which has CM with the maximal order in $\Q(\sqrt{-43})$. This latter point is also one of the 4 ramification points of $X_0(43) \to X_0(43)^+$, and is denoted by $P_0$ in  \cite[\S 4.1]{box2021quadratic}. $X_0(43)$ has infinitely many quadratic points. The curve $X_0(43)^+$ is a positive rank elliptic curve and pulling back rational points along $X_0(43) \to X_0(43)^+$ produces infinitely many quadratic points on $X_0(43)$. There are only 4 points that are not pullbacks of rational points on $X_0(43)^+$, their $j$-invariants can be found in \cite[\S 4.1]{box2021quadratic} and these are denoted $P_1, P_2, P_3$ and $P_4$ in Box's table in \emph{loc. cit.}.

\begin{proof}
\begin{enumerate}
    \item The infinitely many quadratic points coming from pulling back $\Q$-rational points along the map $X_0(43) \to X_0(43)^+$ do not have signature $(0,0)$ by \cref{lem:atkin_lehner_signature}. The unique noncuspidal rational point on $X_0(43)(\Q)$ (i.e. $P_0$) is a ramification point so still satisfies $w_{43}(x) = x$, so \cref{lem:atkin_lehner_signature} is applicable for this point as well, and the signature of that point is also not $(0,0)$. So if an elliptic curve has a $43$-isogeny of signature $(0,0)$, it must correspond to one of the 4 points $P_1, P_2, P_3$ or $P_4$ on $X_0(43)$. We computed the signatures of $P_1, P_2, P_3$ and $P_4$ using the function \path{eps_from_isogeny} in \path{magma_scripts/EpsilonTypes.m}. These were $(12,6), (0,6), (0,0)$ and $(12, 12)$  respectively. In particular, $P_3$ is the only one with signature $(0,0)$. Part (1) now follows since $P_3$ is defined over $K$ and has $j$-invariant $j_{43}$.
    \item The displayed elliptic curve has $j$-invariant equal to $j_{43}$; this was then verified in Sage by computing the $43$-isogeny, its kernel polynomial, and verifying that it is unramified at the primes of $K$ above $43$.
    \item We first claim that there is no elliptic curve over a quadratic field with an isogeny of degree $43$ such that the Galois representation on the kernel of this isogeny is everywhere unramified. Indeed since such an isogeny would have signature $(0, 0)$ the quadratic field would have to be $K$ and the elliptic curve would have to have $j$-invariant equal to $j_{43}$. We computed the action of Galois on the kernel of the unique $K$-rational 43-isogeny of $E_{43}$. We found that the size of the image of Galois was of order 42. Since all elliptic curves with $j$-invariant equal to $j_{43}$ are a quadratic twist of $E_{43}$, we know that the size of the image of Galois is always divisible by $21$. However, the class number of $K$ is 7, so it doesn't admit an unramified character whose order is divisible by 21. This finishes the proof of the claim.
    
    Now suppose that there is an elliptic curve over a quadratic field $K$ that obtains a $43$-torsion point $P$ over an everywhere unramified extension $L$, which we can assume to be Galois over $K$. 
    
    We first observe that
    $\langle P \rangle$ cannot be $K$-rational. Indeed, if it were, the action of
    $\Gal(\overline{K}/K)$ on $\langle P \rangle$ would be given by a character
    $\chi$, which is trivial on $\Gal(\overline{K}/L)$ since $P \in E(L)$, and
    hence factors through $\Gal(L/K)$. As $L/K$ is everywhere unramified, so is
    $\chi$, giving $E/K$ a $43$-isogeny whose kernel carries an everywhere
    unramified Galois action, contrary to the claim.
    
    Hence $E[43]$ is generated by both $P$ and $\sigma(P)$ for some $\sigma \in \Gal(L/K)$. As in the proof of \Cref{cor:main_in_body}, the standard application of the Galois equivariance of the Weil pairing shows that $\zeta_{43} \in L$. But since $L/K$ is assumed to be everywhere unramified, it follows that $\zeta_{43} \in K$; thus $\Q(\zeta_{43}) \subseteq K$. This is a clear contradiction since $[\Q(\zeta_{43}) : \Q] = 42$.
\end{enumerate}
\end{proof}

\begin{remark}
The computation in part~(2) is essential. The signature is an invariant of 
$\lambda^{12}$, not of $\lambda$: by \Cref{prop:refined_signature}, signature $(0,0)$ 
gives only $\lambda^{12}_{I_\fp} = 1$ at the primes $\fp$ above $43$, leaving 
$\lambda_{I_\fp} = \theta_{42}^{\,b}$ with $b \in 7\Z/42$ undetermined. Part~(2) \emph{asserts existence} of an unramified extension, namely the fixed field of 
$\ker\lambda$, and so needs the sufficient condition $b = 0$ on $\lambda$ itself --- 
strictly finer than the signature, hence requiring the computation. Part~(3) instead 
\emph{obstructs existence}, for which a necessary condition suffices: that 
$\lambda^{12}$ be everywhere unramified already places it in $\mathrm{Cl}(K)$, whose 
order $7$ contradicts the divisibility by $21$. The obstruction never descends below 
$\lambda^{12}$, where the signature is faithful.
\end{remark}

\subsection{The prime \texorpdfstring{$p = 73$}{p = 73}}

This prime arises, and we moreover have the following more precise statement.

\begin{proposition}\label{prop:prime_73}
Let $E$ be an elliptic curve over a quadratic field $K$ admitting a
$K$-rational $73$-isogeny of signature $(0,0)$. Then $K = \Q(\sqrt{-31})$ and
$j(E)$ equals \[
  j_1 = \frac{\pm 6561\sqrt{-31} - 1809}{4}\] or \[j_2 = \frac{\pm 218623729131479023842537441 \sqrt{-31}
        - 75276530483988147885303471}{18889465931478580854784}.
\]
Moreover, exactly one of $j_1, j_2$ is the $j$-invariant of such a curve.
\end{proposition}


\begin{proof}
By \cite[Section~4.7]{box2021quadratic}, every quadratic point on $X_0(73)$
except the two points $P_1$ and $P_2$ (in the notation of \emph{loc.\ cit.})
arises from a $\Q$-point on $X_0(73)^+$, so by
\Cref{lem:atkin_lehner_signature} the associated isogeny does not have
signature $(0,0)$. Hence a $K$-rational $73$-isogeny of signature $(0,0)$ can
only correspond to $P_1$ or $P_2$; in particular $K = \Q(\sqrt{-31})$ and
$j(E)$ is one of $j_1, j_2$.

It remains to prove that exactly one of $j_1, j_2$ is realised. Since $73$ is
inert in $K$ and these $j$-invariants are realised by curves that are
semistable at $(73)$, \Cref{lem:00_sig} applies, and for $P_1$ either it or its
dual isogeny has signature $(0,0)$; thus at least one of $j_1, j_2$ is realised.

To see that not both are, note that the dual of the isogeny corresponding to
$P_1$ is given by $w_{73}(P_1)$, which is again a $K$-rational quadratic point
not arising from $X_0(73)^+$, hence lies in $\{P_1, P_2, \sigma(P_1),
\sigma(P_2)\}$, where $\sigma$ generates $\Gal(K/\Q)$. Now $w_{73}(P_1) = P_1$
would give $E$ an endomorphism of degree $73$ and hence CM, and $w_{73}(P_1) =
\sigma(P_1)$ would make $E$ a $\Q$-curve; both are excluded by Box's table.
Therefore $w_{73}(P_1) \in \{P_2, \sigma(P_2)\}$, so the isogenies of $P_1$ and
$P_2$ are dual to one another up to Galois conjugacy. Their isogeny characters
thus multiply to the mod-$73$ cyclotomic character, which is ramified at the
prime above $73$; consequently their signatures are complementary and sum to $(12,12)$, so they cannot both equal $(0,0)$. Hence exactly one of
$j_1, j_2$ is realised.
\end{proof}

\begin{remark}
We were not able to compute the kernel of the 73-isogeny that exists for the two $j$-invariants in \Cref{prop:prime_73}. In principle this is possible via Sage (with the method \path{isogenies_prime_degree}), but this didn't finish after several hours, so we here determined the signature through theoretical means. Since we only have the $j$-invariants and not the actual isogeny, we were not able to determine which of the two $j$-invariants above admit an isogeny of signature $(0,0)$.
\end{remark}

\subsection{The prime \texorpdfstring{$p = 109$}{p = 109}} This prime can be ruled out by recent work of \cite{adzaga2023computing} who determine all quadratic points on $X_0(109)$. By looking at Table 13 of \emph{loc. cit.} we see that all noncuspidal points on this curve are such that the Atkin-Lehner involution $w_{109}$ interchanges the quadratic point on $X_0(109)$ with its Galois conjugate; therefore by \Cref{lem:atkin_lehner_signature} the signature of these isogenies cannot be $(0,0)$.

\section{Results of computation, and how many signatures unconditionally remain}\label{sec:how_many_left}

In this final section we quantify the extent to which \Cref{thm:main} falls short of giving a complete classification of prime degree isogenies of elliptic curves over degree $d$ number fields, and we record the results of running our implementation for small values of $d$.

Recall from \Cref{thm:main} that the only signatures $\eps$ for which the theorem provides no bound on $p$ are those with $\trace\eps\equiv0\Mod6$, other than the two Type $1$ signatures $(0,\ldots,0)$ and $(12,\ldots,12)$, which are dealt with by part (b): for $\trace\eps\equiv6\Mod{12}$ part (c) provides only a divisibility criterion rather than a bound, while for $\trace\eps\equiv0\Mod{12}$ (other than Type $1$) the theorem says nothing. Thus the unconditionally-remaining signatures are exactly those with $\trace\eps\equiv0\Mod6$ (other than Type $1$). The natural question, raised already in the introduction, is how numerous these remaining signatures are among all $5^d$ signatures of degree $d$; in particular, whether the assertion that ``most'' signatures are dealt with is justified, and whether it remains so as $d \to \infty$. We make this precise below. After setting up the relevant group actions (\Cref{def:sig_space,def:gh_action}), we show that the number of signatures one is required to consider is, in the most favourable case, a degree $4$ polynomial in $d$, and that this is best possible (\Cref{prop:sd_count,prop:min_count}); we then show that the proportion of remaining signatures tends to $1/3$ (\Cref{prop:proportion,cor:most}). In particular, for $d \geq 2$, the remaining signatures are always a minority, but never a vanishing one. Under GRH these results can be sharpened: the signatures with $\trace\eps\equiv6\Mod{12}$ then acquire a bound (\Cref{thm:main_with_GRH}), further shrinking the remaining proportion, as we explain in \Cref{subsec:grh_count}.

\subsection{The space of signatures and its symmetries}\label{subsec:sig_space}

\begin{definition}\label{def:sig_space}
Let $k$ be a number field of degree $d$ with Galois closure $K/\Q$, and write $X_k := \Hom(k,K)$ for the set of $d$ embeddings of $k$ into $K$ (this is the index set denoted $\Sigma$ in \cite[Definition 3.3]{banwait2022derickx}). Put $V := \left\{0,4,6,8,12\right\}$. By \cite[Definition 3.3]{banwait2022derickx} a signature is a function $\eps \colon X_k \to V$, $\sigma \mapsto a_\sigma$; we write $\mathcal{F}(X_k, V) = V^{X_k}$ for the set of all $5^d$ signatures. The map
\[ \iota \colon \eps \longmapsto 12 - \eps, \qquad (\iota\eps)(\sigma) := 12 - \eps(\sigma), \]
is an involution of $\mathcal{F}(X_k,V)$, identifying a signature with that of the dual isogeny. We set
\[ \Sigma_k := \mathcal{F}(X_k, V)/\langle \iota \rangle. \]
The only signature fixed by $\iota$ being the constant signature $6\sum_\sigma \sigma$, we have $|\Sigma_k| = \frac{5^d + 1}{2}$.

Finally, for $\eps \in \mathcal{F}(X_k,V)$ we call the tuple $c(\eps) := (n_0, n_1, n_2, n_3, n_4)$, where $n_i := \#\left\{\sigma \in X_k : \eps(\sigma) = v_i\right\}$ for $V = \left\{v_0 < v_1 < \cdots < v_4\right\}$, the {\bf content} of $\eps$. Thus $\sum_i n_i = d$ and $\trace\eps = 4n_1 + 6n_2 + 8n_3 + 12n_4$.
\end{definition}

\begin{definition}\label{def:gh_action}
Let $G := \Gal(K/\Q)$ and let $H := \Aut(k)$ be the group of automorphisms of $k$. The group $G$ acts on $X_k$ on the left by post-composition, $g \cdot \sigma = g \circ \sigma$, and $H$ acts on $X_k$ by pre-composition, $h \cdot \sigma = \sigma \circ h^{-1}$. These two actions commute, and so combine into a single left action of $G \times H$ on $X_k$,
\[ (g,h) \cdot \sigma := g \circ \sigma \circ h^{-1}. \]
Since signatures are functions on $X_k$, this induces a left action of $G \times H$ on $\mathcal{F}(X_k, V)$ by $\big((g,h)\cdot\eps\big)(\sigma) := \eps\big((g,h)^{-1}\cdot\sigma\big)$, which commutes with $\iota$ and hence descends to an action of $G \times H$ on $\Sigma_k$.
\end{definition}

Because the integers $A(\eps,\fq)$, $B(\eps,\fq)$ and $C(\eps,\fq)$ of \cite[Definition 3.8, Corollary 3.10]{banwait2022derickx} are obtained as norms from $K$ to $\Q$, they are invariant under the $G$- and $H$-actions, as well as under $\iota$ (see also \cite[Corollary 4.4]{banwait2022derickx}). Consequently the bound on $p$ furnished by \Cref{thm:main} depends only on the class of $\eps$ in the orbit space $(G\times H)\backslash\Sigma_k$, so that in addressing the uniformity question~\ref{q:uniformity} it suffices to treat one signature in each such orbit. We therefore study the size of $(G\times H)\backslash\Sigma_k$. Note that both factors of $G \times H$ act on $X_k$ by permutations, so the content $c(\eps)$ is a $G\times H$-invariant; the involution $\iota$ reverses content, sending $(n_0,n_1,n_2,n_3,n_4)$ to $(n_4,n_3,n_2,n_1,n_0)$.

\subsection{Two extremal cases}\label{subsec:extremes}

The size of $(G\times H)\backslash\Sigma_k$ depends strongly on the pair $(G,H)$. We treat the two extreme cases: when $k$ is abelian, and when $k$ has the largest possible Galois group.

\begin{proposition}[Abelian case]\label{prop:abelian_count}
If $k/\Q$ is abelian, then $K = k$, the set $X_k$ is identified with $G = H$, the $G$- and $H$-actions on $\Sigma_k$ coincide, and
\[ |(G\times H)\backslash\Sigma_k| \;=\; |\Sigma_k/G| \;\geq\; \frac{5^d + 1}{2d}, \]
with strict inequality for $d > 1$. In particular this quantity grows exponentially in $d$.
\end{proposition}

\begin{proof}
If $k/\Q$ is abelian then it is Galois, so $K = k$ and $X_k = \Hom(k,k) = \Gal(k/\Q) = G$, while $H = \Aut(k) = G$. For $g$ in $G$, viewed as an element of $X_k$, post-composition by $g'$ sends $g \mapsto g'g$, while pre-composition by $h$ sends $g \mapsto g h^{-1}$. Since $G$ is abelian we have $g h^{-1} = h^{-1} g$, so the $H$-action by $h$ coincides with the $G$-action by $h^{-1}$. Hence the $G$- and $H$-orbits on $\Sigma_k$ are identical, and $(G\times H)\backslash\Sigma_k = \Sigma_k/G$. As $|G| = d$, every $G$-orbit has size at most $d$, so the number of orbits is at least $|\Sigma_k|/d = \frac{5^d+1}{2d}$; the inequality is strict for $d > 1$ because, for instance, the constant signatures are fixed by $G$.
\end{proof}

\begin{proposition}[Generic case]\label{prop:sd_count}
Suppose that $G$ acts on $X_k$ as the full symmetric group, that is, $\Gal(K/\Q) \cong S_d$ in its action on the $d$ embeddings. Then two signatures lie in the same $G\times H$-orbit if and only if they have the same content; hence the orbits of $G \times H$ on $\Sigma_k$ are in bijection with content vectors taken up to reversal, and
\[ |(G\times H)\backslash\Sigma_k| \;=\; N(d) \;:=\; \frac12\left(\binom{d+4}{4} + \binom{\lfloor d/2\rfloor + 2}{2}\right) \;=\; \frac{d^4}{48} + O(d^3). \]
\end{proposition}

\begin{proof}
The content is always a $G\times H$-invariant. Conversely, if $G$ acts as $S_d$ on $X_k$ then any two signatures with the same content differ by a permutation of $X_k$, and that permutation is realised by some $g \in G$; the two signatures therefore lie in the same $G$-orbit, a fortiori the same $G \times H$-orbit. Thus on $\mathcal{F}(X_k,V)$ the orbits of $G \times H$ are exactly the fibres of the content map $c$, of which there are $\binom{d+4}{4}$, the number of tuples $(n_0, \ldots, n_4) \in \Z_{\geq 0}^5$ with $\sum_i n_i = d$. Passing to $\Sigma_k$ identifies $\eps$ with $\iota\eps$, hence identifies each content with its reversal; the orbits of $G \times H$ on $\Sigma_k$ are therefore the content vectors up to reversal. By Burnside's lemma their number is $\frac12(T + F)$, where $T = \binom{d+4}{4}$ is the number of content vectors and $F$ is the number fixed by reversal. A content is reversal-invariant precisely when $n_0 = n_4$ and $n_1 = n_3$, that is $2(n_0 + n_1) + n_2 = d$; such tuples are determined by a pair $(n_0, n_1)$ with $n_0 + n_1 \leq \lfloor d/2\rfloor$, of which there are $\binom{\lfloor d/2\rfloor + 2}{2}$. This yields $N(d)$, and the asymptotic follows from $\binom{d+4}{4} = \frac{d^4}{24} + O(d^3)$.
\end{proof}

\subsection{The minimal number of signature classes}\label{subsec:minimal}

The generic case is in fact the most favourable one: it realises the least possible number of signature classes among all fields of degree $d$.

\begin{proposition}\label{prop:min_count}
For every number field $k$ of degree $d$,
\[ |(G\times H)\backslash\Sigma_k| \;\geq\; N(d), \]
with equality whenever $G$ acts on $X_k$ as $S_d$. Since fields of degree $d$ with Galois group $S_d$ exist for every $d$, and are the generic ones, the minimum
\[ \min_{[k:\Q] = d} |(G\times H)\backslash\Sigma_k| = N(d) = \frac{d^4}{48} + O(d^3) \]
is attained, and is a degree $4$ polynomial in $d$ (a quasi-polynomial, owing to the floor in $N(d)$).
\end{proposition}

\begin{proof}
Since the content is a $G\times H$-invariant which is reversed by $\iota$, the assignment $\eps \mapsto \left\{c(\eps), c(\iota\eps)\right\}$ of a content up to reversal is constant on $G\times H$-orbits in $\Sigma_k$. Every content up to reversal is realised by some signature, so this assignment takes exactly $N(d)$ distinct values, by the count in \Cref{prop:sd_count}. Hence $(G\times H)\backslash\Sigma_k$ surjects onto a set of size $N(d)$, giving $|(G\times H)\backslash\Sigma_k| \geq N(d)$. Equality when $G$ acts as $S_d$ is the content of \Cref{prop:sd_count}, and the existence of degree $d$ fields with Galois group $S_d$ is classical.
\end{proof}

\begin{remark}\label{rem:4transitive}
Equality in \Cref{prop:min_count} in fact holds if and only if $G$ acts on $X_k$ as a $4$-transitive permutation group, the binding case being the content $(d-4,1,1,1,1)$; by the classification of finite simple groups the only such actions, apart from $S_d$ and the alternating group $A_d$ ($d \geq 6$), occur for the Mathieu groups. We shall only need the case $G = S_d$.
\end{remark}

\subsection{Most signatures satisfy the trace condition}\label{subsec:proportion}

We now quantify the word ``most''. Write $r_6(d)$ for the number of signatures $\eps \in \mathcal{F}(X_k, V)$ with $\trace\eps \equiv 0 \Mod{6}$; up to the two Type $1$ signatures, these are exactly the signatures left unresolved by \Cref{thm:main}.

Both this count and its analogue under GRH (\Cref{subsec:grh_count}) amount to counting objects whose trace, or some linear form in their content, lies in a prescribed residue class. We isolate the elementary tool for such counts, sometimes called \emph{series multisection}: averaging a root of unity over the residue classes modulo $m$ detects divisibility by $m$.

\begin{lemma}[Series multisection]\label{lem:multisection}
Let $m \geq 1$ be an integer and $\zeta := e^{2\pi i/m}$ a primitive $m$th root of unity.
\begin{enumerate}[label=(\roman*)]
\item For every $n \in \Z$,
\[ \frac{1}{m}\sum_{j=0}^{m-1} \zeta^{jn} = \begin{cases} 1 & \text{if } n \equiv 0 \Mod{m},\\ 0 & \text{otherwise.}\end{cases} \]
\item Consequently, for any finite family of integers $(n_s)_{s \in S}$ and complex weights $(c_s)_{s \in S}$,
\[ \sum_{\substack{s \in S \\ n_s \equiv 0 \,(m)}} c_s \;=\; \frac{1}{m}\sum_{j=0}^{m-1}\,\sum_{s \in S} c_s\,\zeta^{j n_s}. \]
\end{enumerate}
\end{lemma}

\begin{proof}
For (i): if $m \mid n$ then $\zeta^{jn} = 1$ for every $j$ and the sum is $m$; otherwise $\zeta^n \neq 1$ and the geometric sum $\sum_{j=0}^{m-1}(\zeta^n)^j = \big((\zeta^n)^m - 1\big)/(\zeta^n - 1)$ vanishes, since $(\zeta^n)^m = (\zeta^m)^n = 1$. Part (ii) follows by multiplying the indicator in (i) by $c_s$ and summing over $s \in S$, exchanging the order of summation.
\end{proof}

\begin{proposition}\label{prop:proportion}
\begin{enumerate}[label=(\roman*)]
\item Among all $5^d$ signatures,
\[ r_6(d) = \frac{5^d + 2\cdot 2^d}{3}, \]
so that $r_6(d)/5^d \to 1/3$ as $d \to \infty$; moreover $r_6(d)/5^d$ decreases monotonically from $3/5$ to $1/3$.
\item Suppose $G$ acts as $S_d$ on $X_k$. Among the $N(d)$ signature classes of \Cref{prop:sd_count}, the proportion with $\trace \equiv 0 \Mod{6}$ --- equivalently, the proportion of content vectors $(n_0, \ldots, n_4)$ satisfying $2n_1 + n_3 \equiv 0 \Mod{3}$ --- is $\frac13 + O(1/d)$.
\end{enumerate}
\end{proposition}

\begin{proof}
(i) Apply \Cref{lem:multisection}(ii) with $m = 6$, the family $S = \mathcal{F}(X_k, V)$ of all signatures, weights $c_\eps = 1$, and $n_\eps = \trace\eps$; writing $\zeta := e^{2\pi i/6}$,
\[ r_6(d) = \frac16\sum_{j=0}^5 \sum_{\eps} \zeta^{j\trace\eps} = \frac16\sum_{j=0}^5 \Big(\sum_{v \in V} \zeta^{j v}\Big)^{d} = \frac16 \sum_{j=0}^5 H(j)^d, \qquad H(j) = 3 + 2\cos\!\Big(\frac{2\pi j}{3}\Big), \]
the second equality because $\trace\eps = \sum_\sigma \eps(\sigma)$ is a sum over the $d$ embeddings, so the sum over signatures factorises, and the third because $\zeta^{6j} = \zeta^{12j} = 1$ gives $\sum_{v \in V}\zeta^{jv} = 3 + \zeta^{4j} + \zeta^{8j}$ with $\zeta^{4j} + \zeta^{8j} = 2\cos(2\pi j/3)$. The values $H(0),\ldots,H(5)$ are $5,2,2,5,2,2$, whence $r_6(d) = \frac16(2\cdot 5^d + 4\cdot 2^d) = \frac13(5^d + 2\cdot 2^d)$. Dividing by $5^d$ gives $r_6(d)/5^d = \frac13\big(1 + 2(2/5)^d\big)$, a sum of positive terms decreasing in $d$, with the stated limit and the extreme value $3/5$ at $d = 1$.

(ii) The trace of a content vector is $4n_1 + 6n_2 + 8n_3 + 12n_4 \equiv 4n_1 + 2n_3 \Mod{6}$ (since $6n_2$ and $12n_4$ are divisible by $6$), which vanishes modulo $6$ if and only if $2n_1 + n_3 \equiv 0 \Mod{3}$. Let $C_6(d)$ be the number of content vectors satisfying this. By \Cref{lem:multisection}(ii) with $m = 3$ and $\omega := e^{2\pi i/3}$, applied to the integers $2n_1 + n_3$ indexed by the content vectors,
\[ C_6(d) = \frac13\sum_{t=0}^2 A_t(d), \qquad A_t(d) = \!\!\!\sum_{\substack{n_0, \ldots, n_4 \geq 0\\ n_0 + \cdots + n_4 = d}}\!\!\! \omega^{t(2n_1 + n_3)} = [x^d]\, \frac{1}{(1-x)^3}\cdot \frac{1}{1 - \omega^{2t}x}\cdot\frac{1}{1 - \omega^{t}x}, \]
the three undifferentiated factors $(1-x)^{-1}$ accounting for the free variables $n_0, n_2, n_4$, whose coefficients in $2n_1 + n_3$ vanish modulo $3$. The order of the pole of the generating function at $x = 1$ is $3 + [\omega^{2t} = 1] + [\omega^{t} = 1]$, which equals $5$ for $t = 0$ and equals $3$ for $t = 1, 2$. Hence $A_0(d) = \binom{d+4}{4} \sim d^4/24$ dominates, and $C_6(d) = \frac13\binom{d+4}{4}\big(1 + O(1/d)\big)$. Dividing by the total $\binom{d+4}{4}$ of content vectors gives proportion $\frac13 + O(1/d)$; the same holds after passing to content vectors up to reversal, since reversal alters numerator and denominator only by $O(d^2)$ fixed-point terms.
\end{proof}

\begin{corollary}\label{cor:most}
As $d \to \infty$, the proportion of signatures that \Cref{thm:main} leaves unresolved tends to $1/3$; equivalently, \Cref{thm:main} resolves a proportion tending to $2/3$ of all signatures. For $d \geq 2$ the unresolved proportion is at most $11/25$, and it is bounded away from $0$. In particular ``most'' signatures satisfy the trace condition of \Cref{thm:main}, and continue to do so as $d \to \infty$, by an asymptotically determined margin.
\end{corollary}

\begin{proof}
The unresolved signatures form, up to the two Type $1$ signatures, the set counted by $r_6(d)$, of density $1/3$ in $\mathcal{F}(X_k,V)$ by \Cref{prop:proportion}(i); this set is $\iota$-stable, so the same density holds in $\Sigma_k$. The stated bound $11/25$ is the value of $r_6(d)/5^d$ at $d = 2$, the maximum over $d \geq 2$ by the monotonicity in \Cref{prop:proportion}(i).
\end{proof}

\begin{remark}\label{rem:general_k}
For a general pair $(G,H)$ we do not give a closed form for the proportion of remaining classes in $(G\times H)\backslash\Sigma_k$, which depends on the orbit sizes and hence on $(G,H)$. However the remaining signatures form a $G\times H$- and $\iota$-stable subset of $\mathcal{F}(X_k,V)$ of density tending to $1/3$; a Burnside computation recovers the same limiting density after passing to orbits in both extremal cases of \Cref{prop:abelian_count,prop:sd_count} (in the generic case this is \Cref{prop:proportion}(ii)). We therefore expect this proportion to be robust across all $k$.
\end{remark}

\subsection{Computations for small degree}\label{subsec:small_degree}

\Cref{tab:signature_counts} records, for $2 \leq d \leq 10$, the minimal number $N(d)$ of signature classes (computed in the generic case $G = S_d$, where signature classes are content vectors up to reversal), together with the number and proportion of those left unresolved by \Cref{thm:main} (the classes with $\trace\eps\equiv0\Mod6$, other than Type $1$). The proportion visibly approaches the limit $1/3 \approx 0.333$ of \Cref{cor:most}, the convergence being quasi-polynomial here, but geometric for the closed-form proportion $r_6(d)/5^d$ of \Cref{prop:proportion}(i) over all $5^d$ signatures. These counts are produced by the routine \path{count_signatures.py} in our repository. The further reduction available under GRH is tabulated alongside in \Cref{subsec:grh_count}.

\begin{table}[htp]
\begin{center}
\begin{tabular}{|c|c|cc|}
\hline
& & \multicolumn{2}{c|}{$\trace \equiv 0 \ (6)$}\\
$d$ & $N(d)$ & remaining & proportion\\
\hline
$2$ & $9$ & $4$ & $0.444$\\
$3$ & $19$ & $8$ & $0.421$\\
$4$ & $38$ & $16$ & $0.421$\\
$5$ & $66$ & $26$ & $0.394$\\
$6$ & $110$ & $43$ & $0.391$\\
$7$ & $170$ & $64$ & $0.376$\\
$8$ & $255$ & $95$ & $0.373$\\
$9$ & $365$ & $133$ & $0.364$\\
$10$ & $511$ & $185$ & $0.362$\\
\hline
\end{tabular}
\vspace{0.3cm}
\caption{\label{tab:signature_counts}The minimal number $N(d)$ of signature classes over degree $d$ fields (\Cref{prop:min_count}), and the number and proportion of those left unresolved by \Cref{thm:main} (trace $\equiv 0 \Mod 6$), excluding the single Type $1$ class (the signatures $(0,\ldots,0)$ and $(12,\ldots,12)$, which are $\iota$-conjugate and hence one class in $\Sigma_k$). The proportion tends to $1/3$ (\Cref{cor:most}); the smaller count available under GRH is given in \Cref{tab:signature_counts_grh}.}
\end{center}
\end{table}

For $d = 2$ and $3$ we describe the situation in more detail in the following examples.

\begin{example}[$d=2$]
Every quadratic field is Galois with group $C_2 = S_2$, so this is at once the abelian case of \Cref{prop:abelian_count} and the generic case of \Cref{prop:sd_count}; the number of signature classes is $N(2) = 9$. These are the $9$ signatures in the first column of \Cref{tab:isogeny_merel_quadratic}, obtained from the $25$ signatures by the $\eps \sim 12 - \eps$ identification (leaving $13$) and the fact that bounding signature $(a,b)$ also bounds $(b,a)$. \Cref{thm:main} deals with all but $4$ of these (those with trace $\equiv 0 \Mod{6}$ other than Type $1$), in agreement with the entry $N(2) = 9$, remaining $4$, of \Cref{tab:signature_counts}.
\end{example}

\begin{example}[$d=3$]\label{ex:cubic}
Here the question splits into the Galois and non-Galois cubic fields, illustrating the two extremes of \Cref{subsec:extremes}. A non-Galois cubic has $G = S_3$, the generic case, and so by \Cref{prop:sd_count} realises the minimum $N(3) = 19$ of \Cref{prop:min_count}; concretely, these $19$ classes are obtained from the $35 = \binom{3+4}{4}$ content vectors by pairing them up under reversal. A Galois cubic is abelian, with $G = C_3$; the $C_3$-action on $\Sigma_k$ has been implemented in the \path{get_redundant_epsilons} method in \path{sage_code/generic.py}, and yields $23$ equivalence classes, in accordance with the bound $23 > \frac{5^3+1}{2\cdot 3} = 21$ of \Cref{prop:abelian_count}. Every non-Galois class is either equal to a Galois class or the union of two of them, so any set of representatives in the Galois case refines the non-Galois one; the $23$ Galois representatives therefore suffice for both. They are:

    \begin{center}
    \begin{tabular}{|c|c|c|c|c|c|c|c|}
    \hline
    $(0,0,0)$ & $(6,6,6)$ & $(4, 6, 12)$ & $(4, 6, 8)$ & $(0, 0, 6)$ & $(6, 8, 8)$ & $(8, 8, 12)$ & $(0, 12, 6)$\\
    $(6, 6, 4)$ & $(0, 0, 4)$ & $(8, 4, 8)$ & $(0, 6, 4)$ & $(8, 8, 8)$ & $(8, 0, 4)$ & $(4, 0, 12)$ & $(0, 12, 8)$\\
    $(4, 6, 0)$ & $(4, 4, 12)$ & $(12, 0, 0)$ & $(0, 8, 4)$ & $(0, 6, 6)$ & $(8, 0, 6)$ & $(0, 0, 8)$ &\\
    \hline
    \end{tabular}
    \end{center}

    \Cref{thm:main} can bound the isogeny primes for all but nine of the Galois equivalence classes: $(8, 8, 8)$, $(12, 0, 0)$, $(0, 6, 6)$, $(8, 0, 4)$, $(0, 8, 4)$, $(6, 6, 6)$, $(4, 6, 8)$, $(0, 0, 6)$, $(0, 12, 6)$; as non-Galois equivalence classes these collapse to eight, the pair $(8, 0, 4)$, $(0, 8, 4)$ sharing the content $\{0, 4, 8\}$ --- matching the entry $N(3) = 19$, remaining $8$, of \Cref{tab:signature_counts}.  If one could uniformly bound the isogeny primes of these signatures, then one would have established uniform boundedness of isogeny primes for elliptic curves over cubic fields. Such uniform boundedness of isogenies is currently only known to exist for $d = 1$. Running the algorithm for the cases the theorem can deal with we obtain \Cref{tab:isogeny_merel_cubic}; this is analogous to \Cref{tab:isogeny_merel_quadratic} shown in the Introduction, except that here we only show the trace of $\eps$, since the bound from the main theorem only depends on the trace. Obtaining these results for $d = 3$ using auxiliary primes $q < 100$ took just over 3 minutes running on a single core of a 2.6 GHz Intel Core i7 processor.
\end{example}

\begin{table}[htp]
\begin{center}
\begin{tabular}{|c|c|}
\hline
Trace & Isogeny Primes $\geq 13$\\
\hline
\multirow{2}{*}{$0$, $36$} & \multirow{2}{*}{\makecell{13, 17, 19, 23, 29, 31, 37, 41, 43, 53, 61, 67, 73, 79, 97, \\ 109, 137, 163, 193, 229, 241, 409, 457, 463}}\\
 & \\[0.5ex]
$4$, $32$ & $17, 23, 29, 41, 47, 53, 71, 101, 107, 113, 137$\\
$8$, $28$ & $17, 23, 29, 41, 47, 53, 59, 71, 89, 113, 137$\\
$10$, $26$ & $17, 23, 29, 41, 53, 59$\\
$14$, $22$ & $17, 23, 29, 41, 47, 71, 107$\\
$16$, $20$ & $17, 23, 29, 41, 53, 71, 113, 137$\\
$6$, $18$, $30$ & ?\\
$12$ & ?\\
$24$ & ?\\
\hline
\end{tabular}

\vspace{0.3cm}
\caption{\label{tab:isogeny_merel_cubic}Status of strong uniform boundedness of isogeny primes of elliptic curves over cubic fields. Our results depend only on the trace of the signature, the possible values of which are shown in the left column.  The `Isogeny Primes $\geq 13$' column then lists the primes $\geq 13$ that could possibly (but may in fact not) occur as the degree of such an isogeny. The traces $6$, $18$, $30$ (for which \Cref{thm:main}(c) gives only a divisibility criterion) and $12$, $24$ remain to be dealt with unconditionally; under GRH the traces $6$, $18$, $30$ become bounded (see \Cref{subsec:grh_count}).}
\end{center}
\end{table}

\subsection{The cost of the computation}\label{subsec:cost}

\begin{example}[$d=7$]
    Running the implementation of \path{uniform_isogeny_primes.py} for $d=7$ and $q \leq 5$ took around 16 hours running on the same machine as in \Cref{ex:cubic}. The largest prime observed across all values of $\trace \eps$ that \Cref{thm:main} can deal with was $86{,}980{,}477$. We have decided not to show these primes here, as for the $\trace \eps = 0$ case alone we found $2{,}793$ possible isogeny primes! The interested reader is free to clone our repository and run the program with $d = 7$ to see these possible isogeny primes.

    For each auxiliary prime $q$ the algorithm runs through every way $q$ can split in a field of degree $d = 7$, the expensive case being a complete split into seven primes of degree one. For each of those it ranges over all the degree-2 Weil polynomials over $\mathbb{F}_q$, of which there are only about $2\sqrt{q}$ once a trace and its negative are identified — six at $q = 5$. It then forms one matrix for every combination of choices across the seven primes, a sevenfold Cartesian product, so the count is that number raised to the seventh power, here $6^7 \approx 280{,}000$ matrices, each requiring the determinant of a $128 \times 128$ integer matrix with entries dozens of digits long. Adding the next prime, $q = 7$, changes little: the number of Weil polynomials is again six, so the dominant computation is the same size and the overall work grows only modestly, from roughly $0.56$ to $0.69$ million determinants per signature. The picture changes at $q = 11$, the first prime at which the number of Weil polynomials rises — to nine — so that the same sevenfold product jumps to $9^7 \approx 4.8$ million matrices. That single prime then requires more determinant evaluations than the entire $q \le 5$ computation combined, and each is more costly besides, as the integer entries grow with $q$. The cost thus climbs gently through $q = 7$ and then sharply beyond it, which is why the computation becomes intractable once the auxiliary-prime bound is pushed much past $q = 5$.
\end{example}

\appendix

\section{Results conditional on GRH}\label{appendix}

As stated in the Introduction, we can obtain further results if we assume GRH. To make absolutely clear which of our results depend on GRH, all such results have been collected in this Appendix.

\subsection{Strengthening the main theorem}

Here, we provide an extension of part (c) of \Cref{thm:main}.

\begin{theorem}\label{thm:main_with_GRH}
    Let $k$ be a number field of degree $d$, and $E/k$ an elliptic curve admitting a $k$-rational $p$-isogeny of signature $\eps$ for $p$ prime. Write $\trace \eps$ for the sum of the integers in $\eps$. Assume that $\trace \eps \equiv 6 \Mod{12}$. Then, assuming GRH, we have the bound \begin{equation}\label{eq:upperbound}
             p \leq \max\left(\left(10^{9\trace \eps} + 10^{108d}\right)^{2^d}, R_d\right),     \end{equation} where $R_d$ is the largest real root  of the function \[ x - \left( g(x)^{2\trace \eps} + g(x)^{24d}\right)^{2^d} \] and $g(x) = \log(6x) + 9 + \frac{5}{2}(\log\log(6x))^2$.
\end{theorem}

\begin{proof}
    First observe that $\eps$ must contain $6$ to satisfy the condition on the trace, which implies that $p \equiv 3 \Mod{4}$ (see e.g. Table 8 in \cite{banwait2022derickx}, and the discussion after Corollary 3.5 of \emph{loc. cit.} for the definition of sextic signature). Quadratic reciprocity then shows that $p$ splitting in $\Q(\sqrt{-q})$ is equivalent to $q$ being inert in $\Q(\sqrt{-p})$.

    We now seek to find a prime $q$ that splits in $\Q(\sqrt{-p})$ and that is bounded by a slowly-increasing function of $p$. Applying \cite[Theorem 5.1]{bach1996explicit} would --- assuming GRH --- ensure the existence of such a $q$ with $q \leq (4\log p + 10)^2$, but this may have that $q$ is $2$ or $3$, which here we need to exclude.

    We thus instead use \cite[Theorem 1.4]{lamzouri2015conditional} as follows (with the reader being warned that the roles of $p$ and $q$ are swapped between our work and theirs). Let $\chi$ denote the quadratic Dirichlet character associated to the quadratic field $\Q(\sqrt{-p})$, viz. $\legendre{-p}{\cdot}$ (since $p \equiv 3 \Mod{4}$), and recall that $q$ splits in $\Q(\sqrt{-p})$ if and only if $\chi(q) = 1$. We induce $\chi$ to a nonprimitive character $\chi'$ of modulus $6p$ and take the usual convention that $\chi'(n) = 0$ if $\gcd(n,6p) > 1$; we thus have that $\chi'(q) = 1$ if and only if $q$ splits in $\Q(\sqrt{-p})$ \emph{and} $q$ is not $2$ or $3$. Thus, in the context of \cite[Theorem 1.4]{lamzouri2015conditional}, we take $H$ to be the kernel of $\chi'$, which is an index $2$ subgroup of $(\Z/6p\Z)^\times$, and we obtain a prime $q \neq 2,3$ that splits in $\Q(\sqrt{-p})$ such that
    \begin{equation}\label{eq:ineq1}
        q \leq \max\left(10^9, g(p) \right).
    \end{equation}
    For this $q$, we have that $p$ does not split in $\Q(\sqrt{-q})$, so from part (c)(ii) of \Cref{thm:main} we conclude that $p$ divides $B^\ast_{\eps,q}$, and then from \Cref{prop:gen_additive_bound} we obtain 
    \begin{equation}\label{eq:ineq2}
        p \leq (q^{\trace \eps} + q^{12d})^{2^d}.
    \end{equation}     
    The inequalities \ref{eq:ineq1} and \ref{eq:ineq2} contradict each other for large enough $p$, from which we obtain the bound on $p$ as displayed in the statement of the theorem.
\end{proof}

\begin{remark}
Since the upper bound \ref{eq:upperbound} in part (iii) of (c) is obtained by looking at when the inequalities \ref{eq:ineq1} and  \ref{eq:ineq2} contradict each other, it is possible to find which of the two terms in \ref{eq:upperbound} is larger by looking at which of the two terms in \ref{eq:ineq1} is the most restrictive when $q=10^9$ and $p=(q^{\trace \eps} + q^{12d})^{2^d}$. By using the upper bound $(q^{\trace \eps} + q^{12d})^{2^d} \leq (2q^{12d})^{2^d}$ one can show that $10^9$ is the largest term in \ref{eq:ineq1} for $d=1,2,3,4$, implying that $R_d$ is the smallest term in \ref{eq:upperbound} for these values of $d$. By using the lower bound $(q^{\trace \eps} + q^{12d})^{2^d} \geq (q^{12d})^{2^d}$ one can show that $g(p)$ is the largest term in \ref{eq:ineq1} for $d=5$. Since $g(p)$ and $(q^{12d})^{2^d}$ are monotonically increasing in $p$ (respectively $d$), it follows that $g(p)$ remains the largest term for $d > 5$. In particular, $R_d$ is the largest term in \ref{eq:upperbound} for $d \geq 5$.
\end{remark}

\begin{remark}
We comment on the size of the bound \ref{eq:upperbound} and on the role of GRH.

By the previous remark, the operative term is the first for $d \leq 4$ and
$R_d$ for $d \geq 5$; in either case it is doubly-exponentially large in $d$.
For $d \leq 4$ the bound is $(10^{9\trace\eps} + 10^{108d})^{2^d}$, of size
roughly $10^{108 d \cdot 2^d}$; in the case of primary interest --- quadratic
fields, $d = 2$, with the smallest admissible trace $\trace\eps = 6$ --- this
is $(10^{54} + 10^{216})^4 \approx 10^{864}$. For $d \geq 5$, where $R_d$ is
the largest real root of $x - (g(x)^{2\trace\eps} + g(x)^{24d})^{2^d}$ and
$g(x)$ grows like $\log x$, one has $R_d \approx C_d^{C_d}$ with
$C_d = 24 d \cdot 2^d$, again doubly-exponential in $d$. Such bounds are
too large for computation; their value is in establishing effectively that the
set of such $p$ is finite.

The hypothesis of GRH is essential to the method, not merely a sharpening of
constants. It is used exactly once, to invoke
\cite[Theorem 1.4]{lamzouri2015conditional} and thereby produce a prime $q$
with $\chi'(q) = 1$ of size $q \leq \max(10^9, g(p))$, hence essentially of size
$\log p$. It is precisely this smallness that makes \ref{eq:ineq1} and
\ref{eq:ineq2} contradictory for large $p$: substituting such a $q$ into
\ref{eq:ineq2} forces $p \leq g(p)^{O_d(1)}$, which fails once $p$ is large
since $g(p)^{O_d(1)} = o(p)$. Unconditionally, the best known effective bounds
on the least prime with a prescribed splitting behaviour --- whether through
Linnik's theorem or the effective
Chebotarev density theorem of Lagarias--Odlyzko --- give only $q \leq p^{A}$
for an absolute constant $A > 0$, that is, a fixed power of $p$ in place of a
power of $\log p$. Substituting any such bound into \ref{eq:ineq2} gives
$p \leq p^{A\max(\trace\eps, 12d)\,2^d}$, whose exponent is at least $24A$ and
so exceeds $1$; the two inequalities are then mutually consistent and yield no
bound on $p$ at all. GRH is therefore what makes the argument produce a finite
bound in the first place.
\end{remark}

The following is obtained from combining the additive bounds from \Cref{thm:main} and \Cref{thm:main_with_GRH}.

\begin{corollary}
    Let $k$ be a number field of degree $d$, and $E/k$ an elliptic curve admitting a $k$-rational $p$-isogeny of signature $\eps$ for $p$ prime. Write $\trace \eps$ for the sum of the integers in $\eps$. Assume that the trace satisfies one of the following conditions:
    \begin{enumerate}
        \item $\trace \eps \nequiv 0 \Mod{6}$;
        \item $\trace \eps = 0$ or $12d$;
        \item $\trace \eps \equiv 6 \Mod{12}$.
    \end{enumerate}
    Then, assuming GRH, we have 
\[ p \leq \max \left((10^{9\trace \eps} + 10^{108d})^{2^d}, R_d \right),\] where $R_d$ is as in \Cref{thm:main_with_GRH}.
\end{corollary}

\subsection{How many signatures remain under GRH}\label{subsec:grh_count}

In \Cref{sec:how_many_left} we found that, unconditionally, the signatures left unresolved by \Cref{thm:main} are exactly those of trace $\equiv 0 \Mod 6$ (other than Type $1$), a set whose proportion among all $5^d$ signatures tends to $1/3$ (\Cref{cor:most}). The strengthening in \Cref{thm:main_with_GRH} now supplies a bound for the signatures of trace $\equiv 6 \Mod{12}$ as well, so that under GRH the only unresolved signatures are those of trace $\equiv 0 \Mod{12}$ (other than Type $1$). We record here how the counts of \Cref{sec:how_many_left} improve under this hypothesis. The structural results of that section --- the group actions, the count $N(d)$ of signature classes, and its minimality (\Cref{prop:sd_count,prop:min_count}) --- are unconditional and carry over unchanged; only the proportion of \emph{remaining} classes is affected.

\begin{proposition}\label{prop:proportion_grh}
Write $r_{12}(d)$ for the number of signatures $\eps \in \mathcal{F}(X_k, V)$ with $\trace\eps \equiv 0 \Mod{12}$. Then
\[ r_{12}(d) = \frac{5^d + 2\cdot 2^d + 3^d}{6}, \]
so that $r_{12}(d)/5^d \to 1/6$ as $d \to \infty$, decreasing monotonically from $2/5$ (at $d = 1$) to $1/6$. Moreover, if $G$ acts as $S_d$ on $X_k$, then among the $N(d)$ signature classes of \Cref{prop:sd_count} the proportion with $\trace \equiv 0 \Mod{12}$ --- equivalently, the proportion of content vectors $(n_0, \ldots, n_4)$ with $2n_1 + 3n_2 + 4n_3 \equiv 0 \Mod{6}$ --- is $\frac16 + O(1/d)$.
\end{proposition}

\begin{proof}
The argument is that of \Cref{prop:proportion}, with the modulus $12$ in place of $6$. Since every value in $V$ is even, $\trace\eps \equiv 0 \Mod{12}$ if and only if $\tfrac12\trace\eps \equiv 0 \Mod 6$; setting $\delta_\sigma := \eps(\sigma)/2 \in \{0,2,3,4,6\}$ and applying \Cref{lem:multisection}(ii) with $m = 6$ and $\zeta := e^{2\pi i/6}$ gives
\[ r_{12}(d) = \frac16\sum_{j=0}^5 G(j)^d, \qquad G(j) = 2 + 2\cos\!\Big(\frac{2\pi j}{3}\Big) + (-1)^j, \]
since $\sum_{v \in V}\zeta^{jv/2} = 2 + \zeta^{2j} + \zeta^{3j} + \zeta^{4j}$ with $\zeta^{2j} + \zeta^{4j} = 2\cos(2\pi j/3)$ and $\zeta^{3j} = (-1)^j$. The values $G(0), \ldots, G(5)$ are $5, 0, 2, 3, 2, 0$, whence the stated formula and $r_{12}(d)/5^d = \frac16\big(1 + 2(2/5)^d + (3/5)^d\big)$, a sum of positive terms decreasing in $d$. The proportion among content vectors is computed exactly as in \Cref{prop:proportion}(ii): writing $C_{12}(d)$ for the number of content vectors with $2n_1 + 3n_2 + 4n_3 \equiv 0 \Mod 6$, \Cref{lem:multisection}(ii) with $m = 6$ expresses $C_{12}(d)$ as a sum of six coefficient extractions whose dominant term, $t = 0$, is $\binom{d+4}{4} \sim d^4/24$ (a pole of order $5$ at $x = 1$), while every other term has a pole of order at most $4$ and so contributes $O(d^3)$. Hence $C_{12}(d) = \frac16\binom{d+4}{4}\big(1 + O(1/d)\big)$, and dividing by the total $\binom{d+4}{4}$ of content vectors gives proportion $\frac16 + O(1/d)$, unaffected by the passage to content vectors up to reversal.
\end{proof}

\Cref{tab:signature_counts_grh} reproduces the count $N(d)$ of \Cref{tab:signature_counts} together with the number and proportion of classes left unresolved under GRH (trace $\equiv 0 \Mod{12}$) and, for comparison, unconditionally (trace $\equiv 0 \Mod 6$). The GRH column counts a subset of the unconditional one, the difference being exactly the classes with $\trace\eps\equiv6\Mod{12}$ that \Cref{thm:main_with_GRH} now bounds.

\begin{table}[htp]
\begin{center}
\begin{tabular}{|c|c|cc|cc|}
\hline
& & \multicolumn{2}{c|}{$\trace \equiv 0 \ (12)$, i.e.\ GRH} & \multicolumn{2}{c|}{$\trace \equiv 0 \ (6)$, i.e.\ unconditional}\\
$d$ & $N(d)$ & remaining & proportion & remaining & proportion\\
\hline
$2$ & $9$ & $3$ & $0.333$ & $4$ & $0.444$\\
$3$ & $19$ & $4$ & $0.211$ & $8$ & $0.421$\\
$4$ & $38$ & $11$ & $0.289$ & $16$ & $0.421$\\
$5$ & $66$ & $14$ & $0.212$ & $26$ & $0.394$\\
$6$ & $110$ & $28$ & $0.255$ & $43$ & $0.391$\\
$7$ & $170$ & $35$ & $0.206$ & $64$ & $0.376$\\
$8$ & $255$ & $59$ & $0.231$ & $95$ & $0.373$\\
$9$ & $365$ & $73$ & $0.200$ & $133$ & $0.364$\\
$10$ & $511$ & $111$ & $0.217$ & $185$ & $0.362$\\
\hline
\end{tabular}
\vspace{0.3cm}
\caption{\label{tab:signature_counts_grh}The minimal number $N(d)$ of signature classes over degree $d$ fields (\Cref{prop:min_count}), and the number and proportion of those left unresolved by \Cref{thm:main} as strengthened by \Cref{thm:main_with_GRH}, under GRH (trace $\equiv 0 \Mod{12}$) and unconditionally (trace $\equiv 0 \Mod 6$), excluding the single Type $1$ class. The GRH column counts a subset of the unconditional column of \Cref{tab:signature_counts}, the difference being the classes with $\trace\eps\equiv6\Mod{12}$; the two proportions tend to $1/6$ and $1/3$ respectively (\Cref{cor:most,cor:most_grh}).}
\end{center}
\end{table}

The improvement is already visible in small degree. For $d = 2$ the four unresolved classes of \Cref{sec:how_many_left} drop to three under GRH, the signature $(0,6)$ now being bounded, in agreement with the entry $N(2) = 9$, remaining $3$, of \Cref{tab:signature_counts_grh}. For $d = 3$ the nine unresolved Galois classes of \Cref{ex:cubic} drop to five --- the four classes $(6,6,6)$, $(4,6,8)$, $(0,0,6)$, $(0,12,6)$ of trace $\equiv 6 \Mod{12}$ now being bounded --- representing four non-Galois classes, in agreement with $N(3) = 19$, remaining $4$. Correspondingly, in \Cref{tab:isogeny_merel_cubic} the traces $6$, $18$, $30$ left open there acquire under GRH the bound $p \leq 1.01 \times 10^{2592}$ furnished by \Cref{thm:main_with_GRH}.

The following is the GRH counterpart of \Cref{cor:most}, and collects the effect of the hypothesis on the counting.

\begin{corollary}\label{cor:most_grh}
Assume GRH. As $d \to \infty$, the proportion of signatures left unresolved by \Cref{thm:main}, as strengthened by \Cref{thm:main_with_GRH}, tends to $1/6$; equivalently, a proportion tending to $5/6$ of all signatures is resolved. For every $d$ this unresolved proportion is at most $2/5$, and it is bounded away from $0$. Thus GRH raises the resolved proportion of \Cref{cor:most} from $2/3$ to $5/6$, the gain being precisely the signatures of trace $\equiv 6 \Mod{12}$.
\end{corollary}

\begin{proof}
Under GRH the unresolved signatures are, up to the two Type $1$ signatures, those counted by $r_{12}(d)$, of density $1/6$ in $\mathcal{F}(X_k, V)$ by \Cref{prop:proportion_grh}; this set is $\iota$-stable, so the same density holds in $\Sigma_k$. The bound $2/5$ is the value of $r_{12}(d)/5^d$ at $d = 1$, the maximum by the monotonicity in \Cref{prop:proportion_grh}. The final assertion combines this with \Cref{cor:most}, the difference $r_6(d) - r_{12}(d) = \tfrac16(5^d + 2\cdot 2^d - 3^d)$ counting exactly the signatures with $\trace\eps\equiv6\Mod{12}$.
\end{proof}

\subsection{Uniform Momose Type 2}

In this section we treat the special case of \emph{Momose Type 2 isogenies}; that is, whose signature is $(6,\cdots,6)$, and furthermore whose $p$-isogeny character $\lambda$ satisfies $\lambda^{12} = \chi_p^6$. Note that this is not just a special case of part (c) of \Cref{thm:main}, since here we have the equality $\lambda^{12} = \chi_p^6$, and not just equality up to an everywhere unramified character. Consequently we will find a much sharper bound on the primes for which such isogenies can arise.

It was to deal with these possible Momose Type 2 isogenies that Momose required GRH, and different, more analytic, techniques are required to resolve it. Moreover, dealing with these isogenies is the most computationally expensive step of the software package \emph{Isogeny Primes} from the previous work of the authors, since it requires checking a certain condition on Legendre symbols for all primes up to quite large bounds (e.g. for $k = \Q(\zeta_7)^+$ this bound is $5.27 \times 10^{13}$).

The result we explain here provides, conditional upon GRH, an explicit uniform bound $C_d$ for Momose Type 2 isogeny primes in the case that $d$ is odd. The integer $C_d$ is computed algorithmically for each odd $d$ (see \Cref{alg:momose_Type_2}). We first assemble the two ingredients --- a uniform version of Momose's ``Condition C'', and an effective Chebotarev bound --- and only then state and prove the theorem, since both refer to the algorithm that computes $C_d$.

Recall from \cite[Proposition 6.1]{banwait2022derickx} that we have the following necessary condition that Momose Type 2 isogenies must satisfy (going back to what Momose called ``Condition C'' in his paper \cite{momose1995isogenies}).

\begin{proposition}[Condition CC]\label{cond:CC}
Let $k$ be a number field, and $E/k$ an elliptic curve admitting a $k$-rational $p$-isogeny of Momose Type 2. Let $q$ be a rational prime, and $\fq$ a prime ideal of $k$ dividing $q$ of residue degree $f$ satisfying the following conditions:
\begin{enumerate}
    \item $q^f < p/4$;
    \item $f$ is odd;
    \item $q^{2f} + q^f + 1 \not\equiv 0 \Mod{p}$.
\end{enumerate}
Then $q$ does not split in $\Q(\sqrt{-p})$.
\end{proposition}

We remark that this condition can be algorithmically implemented to check if a given prime $p$ is a possible Momose Type~$2$ prime for \emph{any} number field of a given degree $d$ as follows. This may be considered a ``uniform Condition CC'' result.

\begin{algorithm}\label{alg:satisfy_condCC_unif}
    Given the following inputs:
\begin{itemize}
    \item an integer $d \geq 1$;
    \item a rational prime $p$,
\end{itemize}
return True or False as follows.
\begin{enumerate}
    \item Compute the set $F$ of odd positive integers $f \leq d$.
    \item For primes $q \leq 1 + (p/4)^{1/\max(F)}$:
        \begin{enumerate}
            \item if $\legendre{q}{p} = 1$, and $p$ does not divide $q^{2f} + q^f + 1$ for all $f \in F$, then return False and terminate the algorithm.
        \end{enumerate}
    \item If not terminated by now, return True.
\end{enumerate}
\end{algorithm}

Two points in \Cref{alg:satisfy_condCC_unif} merit explanation. First, the range of $q$. The algorithm must refute \Cref{cond:CC} \emph{uniformly}, i.e. for \emph{every} degree $d$ field at once. For an arbitrary such field $k$ and rational prime $q$, the only structure we can rely on is that, $d$ being odd, some prime $\fq \mid q$ has odd residue degree $f \leq d$ (writing $q O_k = \prod_i \fq_i^{e_i}$ we have $\sum_i e_i f_i = d$ odd, so some $f_i$ is odd, and $f_i \leq d$); crucially, we cannot predict which odd $f \leq d$ occurs. A single $q$ therefore refutes \Cref{cond:CC} for all fields simultaneously only if conditions (1) and (3) hold for \emph{every} odd $f \leq d$ --- hence the quantifier ``for all $f \in F$''. Since $q^f$ is increasing in $f$, condition (1) holds for all $f \in F$ exactly when it holds for $f = \max(F)$, i.e. when $q^{\max(F)} < p/4$, which is the stated range. Taking instead all primes $q < p/4$ and testing only those $f$ with $q^f < p/4$ would \emph{not} be sound: in a field where $q$ is inert one has $f = d$ and $q^d \geq p/4$, so condition (1) fails for the only prime above $q$ and such a $q$ cannot exclude that field. Such inert fields genuinely occur (one is constructed by \path{condition_CC_field} in \path{sage_code/type_two_primes.py}), so the tighter range is necessary.

Second, the Legendre symbol. Momose Type 2 primes satisfy $p \equiv 3 \Mod{4}$ \cite[Theorem A]{momose1995isogenies}, so $\Q(\sqrt{-p})$ has discriminant $-p$ and the splitting of $q$ is governed by a single symbol. For an odd prime $q \neq p$, $q$ splits in $\Q(\sqrt{-p})$ if and only if $-p$ is a square modulo $q$, i.e. $\legendre{-p}{q} = 1$. Now
\[ \legendre{-p}{q} = \legendre{-1}{q}\legendre{p}{q} = (-1)^{\frac{q-1}{2}}\legendre{p}{q}, \]
while quadratic reciprocity, together with $p \equiv 3 \Mod{4}$ (so that $\tfrac{p-1}{2}$ is odd), gives $\legendre{p}{q}\legendre{q}{p} = (-1)^{\frac{p-1}{2}\frac{q-1}{2}} = (-1)^{\frac{q-1}{2}}$, hence $\legendre{p}{q} = (-1)^{\frac{q-1}{2}}\legendre{q}{p}$. Substituting, the two sign factors cancel and $\legendre{-p}{q} = \legendre{q}{p}$. Thus $q$ splits in $\Q(\sqrt{-p})$ if and only if $\legendre{q}{p} = 1$, which is exactly the test in Step (2a). The same equivalence holds for $q = 2$: since $-p$ is odd, $2$ splits in $\Q(\sqrt{-p})$ precisely when $-p \equiv 1 \Mod{8}$, i.e. $p \equiv 7 \Mod{8}$, which (given $p \equiv 3 \Mod{4}$) is exactly when $\legendre{2}{p} = 1$.

This algorithm is implemented as \path{satisfies_condition_CC_uniform} in \path{sage_code/type_two_primes.py}, and is used to prove \Cref{thm:momose_type_2_unif}.

We now package the computation of the bound as the following algorithm, whose output $C_d$ is the integer appearing in \Cref{thm:momose_type_2_unif}; its correctness is established in the proof of that theorem below.

\begin{algorithm}\label{alg:momose_Type_2}
    Given an odd integer $d \geq 1$, compute a (conditional on GRH) bound $C_d$ on the Momose Type 2 isogeny primes over degree $d$ number fields as follows.
\begin{enumerate}
    \item Compute the following quantities.
    \begin{itemize}
        \item $S_d$, the largest real root of the function \[ x - \left( (4\log x + 10)^{4d} + (4\log x + 10)^{2d} + 1\right). \]
        \item $F$, the set of odd positive integers $f \leq d$.
        \item $V_d := (4\log S_d + 10)^2$.
        \item $T_d$, the set of prime divisors of the integer \[ \prod_{q \leq V_d} \prod_{f \in F} (q^{2f} + q^f + 1). \]
        \item $P$, the largest $p \in T_d$ for which \Cref{alg:satisfy_condCC_unif} with inputs $d$ and $p$ returns True.
    \end{itemize}
    \item Return $\max\left(P, 4(4\log S_d + 10)^{2d}\right)$.
\end{enumerate}
\end{algorithm}

The set $F$ and the range of $q$ in the product defining $T_d$ are dictated by the same uniformity considerations as in \Cref{alg:satisfy_condCC_unif}: we take \emph{all} odd $f \leq d$ because the residue degree of the relevant prime $\fq$ is not known in advance, and we range over primes $q \leq V_d$ because, by \Cref{eq:bach_sorenson_bound}, $V_d$ bounds the least prime splitting in $\Q(\sqrt{-p})$ once the preliminary bound $p < S_d$ is in force (see the proof of \Cref{thm:momose_type_2_unif} below). Thus $T_d$ is precisely the set of candidate primes \emph{not} already excluded by the analytic argument, each of which is then tested individually with \Cref{alg:satisfy_condCC_unif}.

\begin{theorem}\label{thm:momose_type_2_unif}
Let $d$ be a positive odd integer, and let $p$ be the prime degree of a Momose Type 2 isogeny of an elliptic curve over a degree $d$ number field. Then, assuming GRH, we have $p \leq C_d$, where $C_d$ is the integer output by \Cref{alg:momose_Type_2}.
\end{theorem}

\begin{table}[htp]
\begin{center}
\begin{tabular}{|c|c|}
\hline
$d$ & $C_d$\\
\hline
$3$ & $253{,}507$\\
$5$ & $2.91 \times 10^{22}$\\
$7$ & $3.83 \times 10^{33}$\\
$9$ & $2.09 \times 10^{45}$\\
$11$ & $3.26 \times 10^{57}$\\
\hline
\end{tabular}
\vspace{0.3cm}
\label{tab:c_d}
\caption{The uniform bound $C_d$ on Momose Type $2$ isogeny primes for odd integers $d \leq 11$. Only the value for $d = 3$ has been optimised; see \Cref{rem:pari_bound} for more details.}
\end{center}
\end{table}

To illustrate the size of $C_d$, we compute in \Cref{tab:c_d} the value of $C_d$ for the first five odd integers $d$ (excluding $d = 1$).

We emphasise that this only treats the subset of isogenies of signature type $2$ whose isogeny character is in fact of Momose Type 2. Uniformly dealing with the set which in our previous work was denoted $\TypeTwoNotMomoseBound$ remains to be done. Aside from the computational application of the above result, we hope that the ideas in the proof may stimulate other researchers to consider how a similar result may be obtained for even degree number fields.

\begin{proof}[Proof of \Cref{thm:momose_type_2_unif}]
We show that the integer $C_d$ returned by \Cref{alg:momose_Type_2} bounds the Momose Type 2 primes. Let $k$ be an arbitrary number field of odd degree $d$. The idea of the proof is that, for $p$ sufficiently large, we can always find a prime $q \in \Z$ and a prime ideal $\fq$ in $O_k$ lying over it that does \emph{not} satisfy the necessary implication expressed in \Cref{cond:CC}. Since the degree of $k$ is odd we know that over every prime $q \in \Z$ there exists a prime $\fq$ of odd residue degree $f \leq d$ lying over it. In particular it suffices to find, for all sufficiently large $p$ (depending on $d$), a prime $q$ such that:

\begin{enumerate}
    \item $q^{2f} + q^f + 1 \not\equiv 0 \Mod{p}$ for all odd $f \leq d$;
    \item $q^f < p/4$ for all odd $f \leq d$;
    \item $q$ splits in $\Q(\sqrt{-p})$.
\end{enumerate}

Observe that the first 2 of these conditions are satisfied when $q^{2d} + q^{d} + 1 < p$. To force the final condition, we use Bach and Sorenson's Effective Chebotarev Density Theorem \cite[Theorem 5.1]{bach1996explicit}. This theorem guarantees the existence of a prime $q$ which splits in $\Q(\sqrt{-p})$ such that
\begin{equation}\label{eq:bach_sorenson_bound}
    q \leq (4\log p + 10)^2.
\end{equation}
Thus, if $p$ satisfies
\[ (4\log p + 10)^{4d} + (4\log p + 10)^{2d} + 1 < p,\]
then we are done. This inequality yields impossibility for Type 2 primes $p$ that are larger than the largest real root $S_d$ of the function \[ x - \left( (4\log x + 10)^{4d} + (4\log x + 10)^{2d} + 1\right). \]

While this gives a bound, and we may conclude the proof here, we may improve this bound as follows. Since now we know that $p < S_d$, we again apply Bach and Sorenson's result to obtain the existence of a prime $q$ which splits in $\Q(\sqrt{-p})$ such that $q < (4\log S_d + 10)^2 = V_d$. If we knew that the conditions in (1) were satisfied for this range of $p$ and $q$, then we would have that all conditions (1)-(3) would be satisfied provided that $q^d < p/4$, which would then yield impossibility for $p > 4(4\log S_d + 10)^{2d}$.

We thus consider the case where $p$ divides $q^{2f} + q^f + 1$ for some odd $f \leq d$ and prime $q < V_d$. Write $T_d$ for the set of such $p$s (and note that this set can be readily computed with a computer algebra system, as in \Cref{alg:momose_Type_2}). For each $p \in T_d$ we can check if there is \emph{any} degree $d$ field satisfying \Cref{cond:CC} using \Cref{alg:satisfy_condCC_unif}. Writing $P$ for the largest such $p \in T_d$ for which \Cref{alg:satisfy_condCC_unif} returns True, we conclude that Momose Type $2$ primes for degree $d$ fields must be less than $\max\left(P, 4(4\log S_d + 10)^{2d}\right)$. This is precisely $C_d$, the output of \Cref{alg:momose_Type_2}.
\end{proof}

\begin{remark}
    One may bound $S_d$ in Step (1) of \Cref{alg:momose_Type_2} as follows. We seek the largest real root of $x - f(x)$, where $f(x) = (4\log x + 10)^{4d} + (4\log x + 10)^{2d} + 1$. The lower bound $(4\log x)^{4d} < f(x)$ is immediate, since $f(x) > (4\log x + 10)^{4d} > (4\log x)^{4d}$. For the upper bound, as soon as $x > e^{11} \approx 59874.142$ we have $\log x > 11 > 10$, hence $4\log x + 10 < 5\log x$; the resulting gap between $(5\log x)^{4d}$ and $(4\log x + 10)^{4d}$ then dominates the lower-order terms $(4\log x + 10)^{2d} + 1$, giving $f(x) < (5\log x)^{4d}$. Thus, for $x > e^{11}$, we have \[ (4\log x)^{4d} < f(x) < (5\log x)^{4d}.\] Now the largest real root of $x - (b\log x)^D$ is \[ \exp\left(-DW_{-1}\left(\frac{-1}{Db}\right)\right), \] where $W_{-1}$ is the $-1$ branch of the Lambert $W$ function \cite{corless1996lambert}. Indeed, substituting $x = e^u$, the equation $x = (b\log x)^D$ becomes $u\,e^{-u/D} = 1/b$, i.e. $\left(\tfrac{-u}{D}\right)e^{-u/D} = \tfrac{-1}{Db}$; this has two real solutions, given by the two real branches of $W$, and the larger root $x$ corresponds to $\tfrac{-u}{D} = W_{-1}\!\left(\tfrac{-1}{Db}\right)$. Applying this with $(b, D) = (5, 4d)$ to the upper bound $(5\log x)^{4d}$ provides an explicit upper estimate for $S_d$, and hence a finite interval in which to locate the root numerically.
\end{remark}

\begin{remark}
    One may use \cite[Theorem 1.4]{lamzouri2015conditional} instead of \cite[Theorem 5.1]{bach1996explicit}; this changes \Cref{eq:bach_sorenson_bound} to \[ q \leq \max\left(10^9, \left(\log(p) + 9 + \frac{5}{2}(\log\log(p))^2\right)^2 \right),\]
    which provides an asymptotically better bound. To keep the above argument and algorithm clean we have opted for using Bach-Sorenson's bound, although in the implementation of \Cref{alg:momose_Type_2} (which is \path{momose_type_2_uniform_bound} in \path{sage_code/type_two_primes.py}) we compute both of the corresponding values of $S_d$ in Step (1) and take the minimum.
\end{remark}

The proof of \Cref{thm:momose_type_2_unif} given above, and consequently the bound $C_d$, can be yet further improved for specific values of $d$, by observing that the constants given in \cite[Theorem 5.1]{bach1996explicit} are not best possible for each $d$, and instead using the more precise bounds to be found in Tables 1-3 of \emph{loc. cit.}. This approach can sometimes reduce the bound to the point where one can apply \Cref{alg:satisfy_condCC_unif} for all primes $p$ up to the bound. We illustrate these improvements for $d = 3$.

\begin{example}\label{ex:cubic1}
Running through the first part of the proof of \Cref{thm:momose_type_2_unif}, we obtain $S_3 = 2.4 \times 10^{29}$.
Having found this bound on $p$, observe that $\log p < 68$. If $\log p < 25$, then we would have the bound $p \leq 7.1 \times 10^{10}$. Otherwise, we have that $25 < \log p < 100$. This is relevant because now we can apply the improved bounds from Table~$3$ of Bach and Sorenson's result to conclude that we can find a $q$ which splits in $\Q(\sqrt{-p})$ such that $q \leq (A\log p + 3B + C)^2$ for the values $(A,B,C) = (1.881, 0.34, 5.5)$. Similar to above, this then subsequently yields impossibility for Type 2 primes when $p > 5 \times 10^{24}$.

Since now we know that $p < 5 \times 10^{24}$, we again apply the improved version of Bach and Sorenson's result to obtain the existence of a prime $q$ which splits in $\Q(\sqrt{-p})$ such that $q < 12{,}814$. If we knew that the two conditions in (1) (for $f = 1$ and $f = 3$) were satisfied for this $p$ and $q$, then we would have that all of them were satisfied provided condition (2) was satisfied, which would then yield impossibility for $p > 8.42 \times 10^{12}$.

We thus compute the set $T_3$ of primes $p$ dividing either $q^6 + q^3 + 1$ or $q^2 + q + 1$ for some prime $q < 12{,}814$. There are $3967$ such $p$s, and for each one, we check if there is \emph{any} cubic field $k$ such that the pair $(k,p)$ satisfies \Cref{cond:CC} using \Cref{alg:satisfy_condCC_unif}. The largest such $p$ is $36{,}523$ which is less than $8.39 \times 10^{12}$. We conclude that Type $2$ primes for cubic fields must be less than $8.39 \times 10^{12}$.

This bound is now small enough for one to apply \Cref{alg:satisfy_condCC_unif} for each $p$ up to this bound.  Given the large values here, this was achieved via the optimised PARI/GP implementation \path{partype2primesuniform.gp} found in the \path{gp_scripts} folder of the package, and one finds that $253{,}507$ is the largest such prime.
\end{example}

\begin{remark}
It is an open question whether the bound $253{,}507$ in the above example is sharp; that is, whether there exists an elliptic curve over a cubic field admitting a Momose Type~$2$ isogeny of degree $253{,}507$. However, this bound \emph{is} sharp for the primes which satisfy \Cref{cond:CC} for cubic fields: writing $p = 253{,}507$, the pair $(k,p)$ satisfies \Cref{cond:CC} for any cubic field $k$ in which every rational prime $q < p/4$ is inert. To see this, recall that \Cref{cond:CC} concerns only primes of \emph{odd} residue degree, and that in a cubic field no residue degree exceeds $3$; the relevant residue degrees are thus $f = 1$ and $f = 3$. Since each rational prime $q < p/4$ is inert in $k$, the only prime of $k$ above it has residue degree $3$, so no prime of $k$ of residue degree $1$ lies above a rational prime $q < p/4$. Hence only $f = 3$ can meet the hypotheses (1) and (2) of \Cref{cond:CC}, in which case (1) becomes $q^3 < p/4$, i.e. $q \leq 37$ (the largest prime below $(p/4)^{1/3} \approx 39.87$); the case $f = 1$ --- a rational prime $q < p/4$ with $p \nmid q^2 + q + 1$ --- cannot occur. A direct check shows that every prime $q \leq 37$ fails to split in $\Q(\sqrt{-p})$, so the conclusion of \Cref{cond:CC} holds whenever its hypotheses (1) and (2) do --- condition (3) need not even be examined --- and hence $(k,p)$ satisfies \Cref{cond:CC}.

Such a cubic field may be readily constructed in a few lines of Sage code (see the function \path{condition_CC_field} in \path{sage_code/type_two_primes.py}). The resulting defining polynomial of this cubic field is a monic cubic polynomial whose non-leading coefficients all have more than $27{,}380$ digits. Finding a cubic field satisfying \Cref{cond:CC} for this prime with smallest absolute value of coefficients of its defining polynomial has not been attempted in this work.
\end{remark}

\begin{remark}\label{rem:pari_bound}
    To obtain the values in \Cref{tab:c_d} we have employed all of the improvements illustrated in \Cref{ex:cubic1} \emph{except} for the final application of \Cref{alg:satisfy_condCC_unif} on all primes up to the large bound. This was only just about feasible for $d=3$, since running \path{partype2primesuniform.gp} up to $8.39 \times 10^{12}$ took most of a day on an old laptop.
\end{remark}

\bibliographystyle{alpha}
\bibliography{references.bib}{}

\newcommand{\etalchar}[1]{$^{#1}$}
\begin{thebibliography}{AKMJ{\etalchar{+}}23}

\bibitem[Abr96]{abramovich1996linear}
Dan Abramovich.
\newblock {A linear lower bound on the gonality of modular curves}.
\newblock {\em International Mathematics Research Notices},
  1996(20):1005--1011, 1996.

\bibitem[AKMJ{\etalchar{+}}23]{adzaga2023computing}
Nikola Ad\v{z}aga, Timo Keller, Philippe Michaud-Jacobs, Filip Najman, Ekin
  Ozman, and Borna Vukorepa.
\newblock Computing quadratic points on modular curves {$X_0(N)$}.
\newblock Preprint available online at
  \url{https://arxiv.org/abs/2303.12566.pdf}, 2023.

\bibitem[Ban22]{banwait2021explicit}
Barinder~S. Banwait.
\newblock Explicit isogenies of prime degree over quadratic fields.
\newblock {\em International Mathematics Research Notices}, 2022.
\newblock Available online at \url{https://doi.org/10.1093/imrn/rnac134}.

\bibitem[BD21]{isogeny_primes}
Barinder~S. Banwait and Maarten Derickx.
\newblock {I}sogeny {P}rimes.
\newblock \url{https://github.com/isogeny-primes/isogeny-primes}, 2021.
\newblock Distributed under the
  \href{https://www.gnu.org/licenses/gpl-3.0.html}{GPL v3+ license.}

\bibitem[BD24]{banwait2022derickx}
Barinder~S. Banwait and Maarten Derickx.
\newblock Explicit isogenies of prime degree over number fields.
\newblock {\em Algebra and Number Theory}, 2024.
\newblock To appear; preprint arXiv:2203.06009.

\bibitem[BN15]{bruin2015hyperelliptic}
Peter Bruin and Filip Najman.
\newblock Hyperelliptic modular curves {$X_0(n)$} and isogenies of elliptic
  curves over quadratic fields.
\newblock {\em LMS Journal of Computation and Mathematics}, 18(1):578--602,
  2015.

\bibitem[Box21]{box2021quadratic}
Josha Box.
\newblock Quadratic points on modular curves with infinite {M}ordell--{W}eil
  group.
\newblock {\em Mathematics of Computation}, 90:321--343, 2021.

\bibitem[BPR13]{bilu2013rational}
Yuri Bilu, Pierre Parent, and Marusia Rebolledo.
\newblock Rational points on ${X}_0^+(p^r)$.
\newblock In {\em Annales de l'Institut Fourier}, volume~63, pages 957--984,
  2013.

\bibitem[BS96]{bach1996explicit}
Eric Bach and Jonathan Sorenson.
\newblock Explicit bounds for primes in residue classes.
\newblock {\em Mathematics of Computation}, 65(216):1717--1735, 1996.

\bibitem[CGH{\etalchar{+}}96]{corless1996lambert}
R.~M. Corless, G.~H. Gonnet, D.~E.~G. Hare, D.~J. Jeffrey, and D.~E. Knuth.
\newblock On the {L}ambert {$W$}-function.
\newblock {\em Advances in Computational Mathematics}, 5:329--359, 1996.

\bibitem[Dav08]{david2008caractere}
Agn{\`e}s David.
\newblock {\em {Caract{\`e}re d'isog{\'e}nie et borne uniforme pour les
  homoth{\'e}ties}}.
\newblock Theses, {Universit{\'e} Louis Pasteur - Strasbourg I}, December 2008.

\bibitem[Dav11]{david2011borne}
Agn\`{e}s David.
\newblock Borne uniforme pour les homoth\'{e}ties dans l'image de {G}alois
  associ\'{e}e aux courbes elliptiques.
\newblock {\em Journal of Number Theory}, 131(11):2175--2191, 2011.

\bibitem[Dav12]{david2012caractere}
Agn\`es David.
\newblock Caract\`ere d'isog\'enie et crit\`eres d'irr\'eductibilit\'e.
\newblock Preprint available online at \url{https://arxiv.org/abs/1103.3892},
  2012.

\bibitem[DKSS23]{derickx2019torsion}
Maarten Derickx, Sheldon Kamienny, William Stein, and Michael Stoll.
\newblock Torsion points on elliptic curves over number fields of small degree.
\newblock {\em Algebra Number Theory}, 17(2):267--308, 2023.

\bibitem[DLRNS18]{daniels2018torsion}
Harris Daniels, {\'A}lvaro Lozano-Robledo, Filip Najman, and Andrew Sutherland.
\newblock Torsion subgroups of rational elliptic curves over the compositum of
  all cubic fields.
\newblock {\em Mathematics of Computation}, 87(309):425--458, 2018.

\bibitem[Edi93]{edixhoven1993rational}
Bas Edixhoven.
\newblock Rational torsion points on elliptic curves over number fields.
\newblock {\em S{\'e}minaire Bourbaki}, 782:209--227, 1993.

\bibitem[Fre94]{frey1994curves}
Gerhard Frey.
\newblock Curves with infinitely many points of fixed degree.
\newblock {\em Israel Journal of Mathematics}, 85:79--83, 1994.

\bibitem[FS15]{freitas2015criteria}
Nuno Freitas and Samir Siksek.
\newblock Criteria for {Irreducibility} of mod $p$ {Representations} of {Frey}
  {Curves}.
\newblock {\em Journal de Th\'eorie des Nombres de Bordeaux}, 27(1):67--76,
  2015.

\bibitem[GR14]{gaudron2014polarisations}
{\'E}ric Gaudron and Ga{\"e}l R{\'e}mond.
\newblock Polarisations et isog{\'e}nies.
\newblock 2014.

\bibitem[GR23]{gaudron2023nouveaux}
Eric Gaudron and Ga{\"e}l R{\'e}mond.
\newblock {\em Nouveaux th{\'e}or{\`e}mes d'isog{\'e}nie}.
\newblock Soci{\'e}t{\'e} math{\'e}matique de France, 2023.

\bibitem[Kam92]{kamienny1992torsion}
Sheldon Kamienny.
\newblock Torsion points on elliptic curves and $q$-coefficients of modular
  forms.
\newblock {\em Inventiones mathematicae}, 109(1):221--229, 1992.

\bibitem[LLS15]{lamzouri2015conditional}
Youness Lamzouri, Xiannan Li, and Kannan Soundararajan.
\newblock Conditional bounds for the least quadratic non-residue and related
  problems.
\newblock {\em Mathematics of Computation}, 84(295):2391--2412, 2015.

\bibitem[LR16]{lozano2016ramification}
{\'A}lvaro Lozano-Robledo.
\newblock Ramification in the division fields of elliptic curves with potential
  supersingular reduction.
\newblock {\em Research in Number Theory}, 2:1--25, 2016.

\bibitem[LV14]{larson_vaintrob_2014}
Eric Larson and Dmitry Vaintrob.
\newblock Determinants of subquotients of {G}alois representations associated
  with abelian varieties.
\newblock {\em Journal of the Institute of Mathematics of Jussieu},
  13(3):517–559, 2014.

\bibitem[Maz77]{Mazur3}
Barry Mazur.
\newblock Modular curves and the {E}isenstein ideal.
\newblock {\em Publications Math{\'e}matiques de l'Institut des Hautes
  {\'E}tudes Scientifiques}, 47(1):33--186, 1977.
\newblock With an appendix by Barry Mazur and Michael Rapoport.

\bibitem[Maz78]{mazur1978rational}
Barry Mazur.
\newblock Rational isogenies of prime degree.
\newblock {\em Inventiones mathematicae}, 44(2):129--162, 1978.
\newblock With an appendix by Dorian Goldfeld.

\bibitem[Mer96]{merel1996bornes}
Lo{\"\i}c Merel.
\newblock Bornes pour la torsion des courbes elliptiques sur les corps de
  nombres.
\newblock {\em Inventiones mathematicae}, 124(1-3):437--449, 1996.

\bibitem[MJ22]{michaud2022elliptic}
Philippe Michaud-Jacobs.
\newblock On elliptic curves with $p$-isogenies over quadratic fields.
\newblock {\em Canadian Journal of Mathematics}, pages 1--20, 2022.

\bibitem[Mom95]{momose1995isogenies}
Fumiyuki Momose.
\newblock Isogenies of prime degree over number fields.
\newblock {\em Compositio Mathematica}, 97(3):329--348, 1995.

\bibitem[MW93]{masser1993isogeny}
David Masser and Gisbert W{\"u}stholz.
\newblock Isogeny estimates for abelian varieties, and finiteness theorems.
\newblock {\em Annals of mathematics}, 137(3):459--472, 1993.

\bibitem[Par99]{parent1999bornes}
Pierre Parent.
\newblock Bornes effectives pour la torsion des courbes elliptiques sur les
  corps de nombres.
\newblock {\em Journal für die reine und angewandte Mathematik}, 1999.

\bibitem[Par18]{parent2018heights}
Pierre Parent.
\newblock Heights on squares of modular curves.
\newblock {\em Algebra \& Number Theory}, 12(9):2065--2122, 2018.

\bibitem[Poo23]{poonen2023rational}
Bjorn Poonen.
\newblock {\em Rational points on varieties}, volume 186.
\newblock American Mathematical Society, 2023.

\bibitem[Ser72]{serre_prop_galois}
Jean-Pierre Serre.
\newblock Propri\'et\'es {G}aloisienne des points d'ordre fini des courbes
  elliptiques.
\newblock {\em Inventiones Mathematicae}, 15:259--331, 1972.

\bibitem[Tri15]{arnav15}
Arnav Tripathy.
\newblock The symmetric power and {\'e}tale realisation functors commute.
\newblock Preprint, {arXiv}:1502.01104 [math.{AG}] (2015), 2015.

\end{thebibliography}

\end{document}